\documentclass[a4paper,11pt,reqno]{article}

\usepackage{tikz-cd} 

\usepackage{relsize}
\usepackage{cite}
\usepackage{color}
\usepackage{hyperref}
\hypersetup{
  colorlinks   = true, 
  urlcolor     = green, 
  linkcolor    = black,
  citecolor   = red 
}
\usepackage{amsmath,amsthm,amssymb,mathtools}

\usepackage[margin=2.6cm]{geometry}

\makeatletter
\usepackage{comment}
\let\wfs@comment@comment\comment
\let\comment\@undefined

\usepackage{changes}
\let\wfs@changes@comment\comment
\let\comment\@undefined

\newcommand\comment{%
    \ifthenelse{\equal{\@currenvir}{comment}}
    {\wfs@comment@comment}
    {\wfs@changes@comment}%
}

\definechangesauthor[name=Matteo, color=red]{MAT}
\definechangesauthor[name=Martino, color=blue]{MAR}
\definechangesauthor[name=Eimear, color=violet]{EIM}
\usepackage{stmaryrd}

\theoremstyle{definition}
\newtheorem{theorem}{Theorem}[section]
\newtheorem{definition}[theorem]{{{Definition}}}
\newtheorem{example}[theorem]{{{Example}}}

\newtheorem{remark}[theorem]{{{Remark}}}

\newtheorem{corollary}[theorem]{{{Corollary}}}
\newtheorem{proposition}[theorem]{{{Proposition}}}
\newtheorem{lemma}[theorem]{{{Lemma}}}

\newcommand{\numberset}{\mathbb}

\newcommand{\F}{\numberset{F}}

\newcommand{\HHH}{\mathcal{H}}
\newcommand{\C}{\mathcal{C}}
\newcommand{\mS}{\mathcal{S}}

\newcommand{\mC}{\mathcal{C}}
\newcommand{\mP}{\mathcal{P}}
\newcommand{\mG}{\mathcal{G}}

\newcommand{\mD}{\mathcal{D}}
\newcommand{\mU}{\mathcal{U}}

\renewcommand{\mG}{\mathcal{G}}

\newcommand{\wt}{\textnormal{wt}}

\newcommand{\Fq}{\F_q}
\newcommand{\Fm}{\F_{q^m}}

\newcommand{\rk}{\textnormal{rk}}

\DeclareMathOperator{\GL}{GL}

\DeclareMathOperator{\PG}{PG}

\newcommand{\Fmk}{[n,k]_{q^m/q}}

\providecommand{\keywords}[1]
{
  	
  \noindent\textbf{{\small \textbf{Keywords.}}} #1
}

\title{Saturating systems and the rank-metric covering radius}

\usepackage{authblk}
\author[1]{Matteo Bonini}
\affil[1]{Aalborg Universitet, Department of Mathematical Sciences, Aalborg, Denmark}

\author[2]{Martino Borello}
\affil[2]{Universit\'e Paris 8, Laboratoire de G\'eom\'etrie, Analyse et Applications, LAGA,
Universit\'e Sorbonne Paris Nord, CNRS, UMR 7539, France}

\author[3]{Eimear Byrne}
\affil[3]{School of Mathematics and Statistics, University College Dublin, Ireland}

\date{}

\begin{document}

\maketitle

\begin{abstract}
    We introduce the concept of a rank-saturating system and outline its correspondence to a rank-metric code with a given covering radius.
    We consider the problem of finding the value of $s_{q^m/q}(k,\rho)$, which is the minimum $\F_q$-dimension of a $q$-system in $\F_{q^m}^k$ that is rank-$\rho$-saturating. This is equivalent to the covering problem in the rank metric. We obtain upper and lower bounds on $s_{q^m/q}(k,\rho)$ and evaluate it for certain values of $k$ and $\rho$. We give constructions of rank-$\rho$-saturating systems suggested from geometry. 
\end{abstract}



\tableofcontents

\section*{Introduction}

The relationships between linear codes and sets of points in finite geometries have long been exploited by researchers \cite{calderbank1986geometry,dodunekov1998codes,alfarano2019geometric,davydov1995constructions,davydov2000saturating,1930-5346_2011_1_119}. Indeed, the MDS conjecture was first posed by Segre as a problem on arcs in finite geometry \cite{segre1955curve}. A generator matrix or parity check matrix of a linear code can be constructed from a multiset of projective points. Supports of codewords correspond to complements of hyperplanes in a fixed projective set. This connection yields a `dictionary' between these two fields, which allows one to apply methods from one domain to the other. This approach has been taken in constructing codes with bounded {\em covering radius}, related to {\em saturating sets} in projective space. 

The geometry of rank-metric codes has recently been investigated \cite{alfarano2021linear,randrianarisoa}: rank-metric codes correspond to $q$-systems and linear sets. 
In this paper, we exploit this relationship further: we introduce the notion of a rank-saturating system in correspondence with a rank-metric covering code. 
    
\bigskip
    
    The covering radius of a code is the least positive integer $\rho$ such that the union of the spheres of radius $\rho$ about each codeword is equal to the full ambient space. This fundamental coding theoretical parameter has been widely studied for codes in respect of the Hamming metric \cite{brualdiplesswil,coveringcodes,davydov1995constructions,1930-5346_2011_1_119,davydov2000saturating,denaux2021constructing,davydov2003saturating}, but very few papers on the subject have appeared in the literature on rank-metric codes \cite{byrne2017covering,gadouleauphd}. The covering radius is an indicator of combinatorial properties of a code, such as {\em maximality} and is an invariant of code equivalence. It also gives a measure of its error-correcting capabilities, via a determination of the maximal weight of a correctable error. Several other communication problems can be expressed in terms of covering problems for the Hamming metric \cite{coveringcodes}. For the rank metric, the covering radius of a code has connections with min-rank problems, such as those that arise in {\em index-coding} \cite{bycalderini}. However, as far as the authors are aware, covering radius problems in the rank metric have not yet been approached from a geometric viewpoint.
    
A set $\mS$ of points in the projective space $\PG(k-1,q)$ is call $\rho$-\emph{saturating} if every point of $\PG(k-1,q)$ lies in a projective subspace spanned by $\rho+1$ points of $\mS$ and $\rho$ is the least integer with this property. The equivalence classes of such sets are in bijection with equivalence classes of {\em covering codes}: if a $\rho$-saturating set $\mS$ is identified with the columns of a parity check matrix of a code $\C$, then $\C$ has (Hamming) covering radius $\rho+1$. This yields an interesting connection between coding theory and finite geometry. A key question in this topic concerns the minimal cardinality of a saturating set for fixed $q,k$, and $\rho$. Translated to codes, this asks what the shortest length of an $\F_q$-code of redundancy $k$ and covering radius $\rho+1$ is. A related problem is to obtain bounds on this number and to give constructions of codes or saturating sets that meet these bounds. Geometric methods to these problems have been considered in \cite{davydov2003saturating,denaux2021constructing,davydov2000saturating,ughi1987saturated,davydov1995constructions}, wherein the two main approaches involve constructions using (1) {\em cutting} (or {\em strong}) blocking sets and (2) {\em mixed subgeometries}.

Rank-metric codes have been a source of intense research activity over the last number of years \cite{byrne2017covering,byrne2020partition,gadouleauphd,neriMRD,delsarte1978bilinear,gabidulin1985theory,roth1991maximum,lunardon2018generalized,sheekey2016new}. In this paper, we focus on rank-metric codes that are $\F_{q^m}$-linear subspaces of $\F_{q^m}^{n}$. While there exists a more general description of rank-metric codes simply as linear spaces of matrices, the restriction to the $\F_{q^m}$-linear subspaces of $\F_{q^m}^{n}$ has more immediate connections to finite geometry \cite{alfarano2021linear,randrianarisoa}. In this paper, we introduce the notion of an $\Fmk$ {\em rank}-$\rho$-{\em saturating system}. In analogy with codes for the Hamming metric, it turns out that a rank-$\rho$-saturating system corresponds to a linear code of rank-metric covering radius $\rho$. Such codes have the property that every element of the ambient space is within rank distance at most $\rho$ to some codeword. In our analysis, we will use the notion of an $\Fmk$ system, which is simply an $n$-dimensional $\Fq$-subspace of $\F_{q^m}^k$ whose $\F_{q^m}$-span is the full space and is a $q$-analogue of a projective system. Such $q$-{\em systems} have been used already in \cite{alfarano2021linear} and \cite{randrianarisoa} to describe geometric aspects of rank-metric codes.
Then an $\Fmk$ rank-$\rho$-saturating system is one whose associated {\em linear set} is a $(\rho-1)$-saturating set in $\PG(k-1,q^m)$.

We write $s_{q^m/q}(k,\rho)$ to denote the minimum $\F_q$-dimension of a rank-$\rho$-saturating system in $\F_{q^m}^k$. In this paper, we show that
\begin{align}\label{eq:mainbd}
    \left\lceil \frac{mk}{\rho}\right\rceil-m+\rho\leq s_{q^m/q}(k,\rho)\leq m(k-\rho)+\rho,
\end{align}
for all $q>2$, and a slightly different lower bound for $q=2$.
While the lower bound of \eqref{eq:mainbd} arises from a combinatorial observation, the upper bound is constructive. Furthermore, using the notion of a {\em linear cutting blocking set} \cite{alfarano2020three}, as well as constructions from subgeometries, we obtain sharper upper bounds for specific parameters by constructing rank-saturating systems.

\bigskip

This paper is organised as follows. In Section \ref{sec:back} we outline some background preliminaries. In Section \ref{sec:ranksat} we introduce the notion of a rank-saturating system, give equivalent characterizations of such systems, and outline the connection to the rank covering radius of a code. In Section \ref{sec:bounds} we give upper and lower bounds on the minimum $\F_q$-dimension of a rank-$\rho$-saturating system. In almost all cases, the bounds we establish turn out to be independent of $q$.
The concept of a rank-saturating system allows us to extend in Section \ref{sec:const} classical constructions for the Hamming metric to the rank metric. In particular, we adopt two different approaches:
one construction arises from \emph{linear} cutting blocking sets, first introduced in \cite{alfarano2021linear}, and the other uses subgeometries. In the final section, we 
list some cases for which $s_{q^m/q}(k,\rho)$ is completely determined. 
\section{Background}\label{sec:back}

Throughout this paper, $q$ will denote a fixed prime power, while $m,n,k$ will denote positive integers such that $n\leq km$ and $k\leq n$. We will write $\rho$ to denote a positive integer in $\{1,\ldots,\min\{k,m\}\}$. Vectors will, as a rule, be column-vectors (unless specified otherwise). We write $[n]$ to denote the set $\{1,\ldots,n\}$.\\

Let $\sim$ denote the equivalence relation on the non-zero elements of $\Fq^k$ defined by $u\sim v$ if and only if $u=\lambda v$
for some nonzero element~$\lambda \in \F_q$. The \emph{projective space} $\PG(k-1,q)$ with underlying vector space $\F_q^k$ is the set of equivalence classes for $\sim$, which are called \emph{points}. For a subspace $V\subseteq \Fq^k$, the corresponding collection of points in $\PG(k-1,q)$ form a \emph{projective subspace} of $\PG(k-1,q)$. If $\dim V=k-1$, the corresponding subspace is called a \emph{hyperplane}.\\

For integers $0\le k \le n$ and a prime power $q$, the \emph{Gaussian binomial coefficient}
$$\left[\begin{array}{c}
	         n\\
	         k
	    \end{array}\right]_q:= \prod_{j=0}^{k-1} \frac{q^n-q^{j}}{q^k-q^{j}},$$
denotes the number of $k$-dimensional subspaces of an $n$-dimensional space over $\F_q$. 

\subsection{Linear codes}

Let us start with some basic definitions of coding theory. Classically applied in noisy channel communication, code elements are often called words and therefore commonly represented as row vectors. In this paper we will mainly consider the codes and systems for the rank metric, but we will point out some relations with the more classical Hamming-metric case. An $\F_{q^m}$-linear code of length $n$ is an $\F_{q^m}$-linear space of $\F_{q^m}^{1\times n}$.

\begin{definition} 
Let $u=(u_1,\ldots,u_n)$ and  $v=(v_1,\ldots,v_n )$ in  $\F_{q^m}^{1\times n}$.
\begin{enumerate}
    \item The \emph{Hamming distance} between $u$ and $v$ in  $\F_{q^m}^{1\times n}$ is defined to be the number of coordinates in which they differ, that is: 
$d_H(u,v):=|\{i\in [n] : u_i\neq v_i\}|.$ The \emph{Hamming weight} of $u$ is  
$\wt_{H}(u):=d_{H}(u,0)$. An $[n,k,d]_{q^m}$ \emph{Hamming-metric} \emph{code} $\C$ is a $k$-dimensional $\F_{q^m}$-subspace of $\F_{q^m}^{1\times n}$ such that
$d=d_{H}(\mC):=\min \{ d_{H}(c,c') : c, c' \in \mC, c \neq c'\}.$
If the minimum distance of $\mC$ is not known or is not relevant, we refer to it as an $[n,k]_{q^m}$ code. 
\item The \emph{rank distance} between $u$ and $v$ in $\F_{q^m}^{1\times n}$ is defined to be the $\Fq$-dimension of the vector space spanned by the differences of their coordinates, that is: $ d_{\rk}(u,v):=\dim_{\Fq}\langle u_i-v_i : i \in [n]\rangle_{\Fq}.$ The \emph{rank weight} of $u$ is $\wt_{\rk}(u):=d_{\rk}(u,0)$. 
An $[n,k,d]_{q^m/q}$ \emph{rank-metric code} $\C$ is a $k$-dimensional $\F_{q^m}$-subspace of $\F_{q^m}^{1\times n}$ such that $d=d_{\rk}(\mC):=\min \{ d_{\rk}(c,c') : c, c' \in \mC, c \neq c'\}.$
If the minimum distance $\mC$ is not known or is not relevant, we refer to it as an $[n,k]_{q^m/q}$ code. 
\end{enumerate}
\end{definition}
An $[n,k]_{q^m}$ or an $[n,k]_{q^m/q}$ code $\mC$ is often described in terms of a \emph{generator matrix} $G\in \F_{q^m}^{k\times n}$, which is a full-rank matrix whose rows generate $\C$.   
    The \emph{dual code} of $\mC$ is defined to be:
    \[
      \mC^\perp:=\{ v \in \F_{q^m}^{1\times n} : v\cdot c = 0 \, \, \forall\: c \in \mC \},
    \]
    where for all $x=(x_1,\ldots,x_n),y=(y_1,\ldots,y_n) \in \F_{q^m}^{1\times n}$ we have
    $x\cdot y:= \sum_{j=1}x_jy_j$.
    
Let $\C, \C'$ be a pair of $\F_{q^m}$-linear codes. We say that $\C$ and $\C'$ are equivalent with respect to the Hamming metric if there exists a monomial matrix $M\in\F_{q^m}^{n\times n}$ and a pair of generator matrices $G,G'$ of $\C,\C'$, respectively, satisfying $G'=GM$.
We say that $\C$ and $\C'$ are equivalent with respect to the rank metric, if there exists $A\in\GL_n(q)$ and a pair of generator matrices $G,G'$ of $\C,\C'$, respectively, satisfying $G'=GA$.

We will generally require that the codes we study are {\em nondegenerate} in the following sense.

\begin{definition}
 An $[n,k]_{q^m}$ code $\C$ is Hamming-metric nondegenerate if for every $i\in [n]$ there exists $c\in\C$ such that $c_i\neq 0$.
An $\Fmk$ code $\mC$ is rank-metric nondegenerate if the $\F_q$-span of the columns of any generator matrix of $\C$ has $\Fq$-dimension $n$.
\end{definition}

Note that if a code is degenerate, then it can be isometrically embedded in an ambient space of smaller dimension.

It was shown in \cite[Proposition 3.2]{alfarano2021linear} that $\mC$ is rank-metric nondegenerate if and only if for every $A \in \GL_n(q)$, the code $\mC \cdot A$ is Hamming-metric nondegenerate. Note that, as already observed in \cite[Corollary~6.5]{jurrius2017defining}, nondegenerate rank-metric $[n,k]_{q^m/q}$ codes may exist only if $n\leq mk$. 

\begin{definition}
An $[n,k]_{q^m}$ code $\C$ is \emph{projective} if $d_H(\C^\perp)\geq 3$. 
We define a \emph{projectivisation} of a code $\C$ to be a punctured code $\C^*$ of $\C$ of maximal length such that $d_H((\C^*)^\perp)\geq 3$.  
\end{definition}

A code is called {\em projective} if and only if no generator matrix has two linearly dependent columns. In a projectivisation one erases the minimum number of columns to obtain a projective code. Any pair of codes obtained by projectivization are equivalent. For this reason, it makes sense to talk about {\em the} projectivisation of a code.

\begin{definition}
Let $\mC \leq \F_{q^m}^{1\times n}$.
The Hamming-metric covering radius of $\mC$ is:
$$\rho_{H}(\mC):= \max \{ \min \{ d_{H}(x,c) : c \in \C\} : x \in \F_{q^m}^{1\times n}\}.$$
   The rank-metric covering radius of $\mC$ is:
$$\rho_{\rk}(\mC):= \max \{ \min \{ d_{\rk}(x,c) : c \in \C\} : x \in \F_{q^m}^{1\times n}\}.$$
\end{definition}
 More generally, with respect to an arbitrary distance function $d$ on $\F_{q^m}^{1\times n},$ the covering radius of a code $\mC \leq \F_{q^m}^{1\times n}$ is the minimum value $\rho$ such that the union of the balls of radius $\rho$ about each codeword, with respect to the distance function $d$, is equal to the full ambient space $\F_{q^m}^{1\times n}$. Hamming-metric (respectively rank-metric) covering radius is an invariant of Hamming-metric (respectively rank-metric) code equivalence.\\

We summarize some well-known results on the covering radius (c.f. \cite{byrne2017covering, coveringcodes}). Similar statements holds for any distance function, but we state them for the rank metric.

\begin{lemma} \label{rd}
  Let $\mC,\mD \leq \F_{q^m}^{1\times n}$ be a pair of rank-metric codes. The following hold.
  \begin{itemize}
\item[{\rm (a)}] If 
$\C \subseteq \mD$, then $\rho_{\rk}(\C) \ge \rho_{\rk}(\mD)$.
\item[{\rm (b)}] \label{p3} If 
$\C \subsetneq \mD$, then $\rho_{\rk}(\C) \ge d_{\rk}(\mD)$.
\item[{\rm (c)}] \label{p4} If $\C \notin \{\{0\},\; \F_{q^m}^{1\times n}\}$, then $d_{\rk}(\C)-1 < 2\rho_{\rk}(\C)$.
  \end{itemize}
 \end{lemma}
 
 An $[n,k,d]_{q^m/q}$ code is called \emph{maximal} if it is not strictly contained in any (possibly non-linear) code $\mD \subseteq \F_{q^m}^{1\times n}$ such that $d_{\rk}(\mD)=d$. Clearly a cardinality-optimal code is also maximal.

\begin{lemma}[The Supercode Lemma, \cite{coveringcodes}] \label{max}
 Let $\mC$ be an $[n,k,d]_{q^m/q}$ code with 
 $|\mC| \ge 2$. Then $\mC$ is 
 maximal if and only if $\rho_{\rk}(\mC) \le d-1$.
\end{lemma}

\begin{example}
   Let $\alpha=(\alpha_j: j \in [n])\in \F_{q^m}^{1\times n}$ have rank weight $n$ over $\Fq$ and let $i$ be a positive integer satisfying $(i,m)=1$.  
   An $[n,k,n-k+1]_{q^m/q}$ code with generator matrix 
   \[G_{n,k,i} = \left( \alpha_j^{q^{i(t-1)}} \right)_{t \in [k],j \in [n]}\] is called as a \emph{generalized Gabidulin code}.  
   We denote it by $\mG_{n,k,i}(\alpha)$. Its dual code is a generalized Gabidulin code $\mG_{n,n-k,i}(\beta)$, for some $\beta \in \F_{q^m}$.
   Such codes meet the rank-metric Singleton bound and are hence maximal, being optimal. 
   Therefore, from the Supercode Lemma, we have 
   $\rho_{\rk}(\mG_{n,k,i}(\alpha)) \leq n-k$. On the other hand, $\mG_{n,k,i}(\alpha)\lneq \mG_{n,k+1,i}(\alpha)$ and so from Lemma \ref{rd}, 
   we have $\rho_{\rk}(\mG_{n,k,i}(\alpha)) \geq d_{\rk}(\mG_{n,k+1,i}(\alpha))=n-k$. 
   It follows that $\rho_{\rk}(\mG_{n,k,i}(\alpha)) = n-k$.
\end{example}

We recall the Dual Distance and External Distance bounds for rank-metric codes \cite{delsarte1973four,byrne2017covering}, which we state for $\Fm$-linear rank-metric codes. The {\em external distance} of an
$\Fmk$ rank-metric code $\C$ is defined to be:
\[s_\rk(\C):=|\{ \wt_{\rk}(c): c\in \C, \,  c \neq 0\}|.\]

\begin{theorem}\label{thm:edb}
     Let $\C$ be a $\Fmk$ rank-metric code. Then the following hold:
     \begin{enumerate}
         \item 
           $\rho_{\rk}(\C^\perp)\leq s_\rk(\C)$ (external distance),
         \item 
           $\rho_{\rk}(\C^\perp)\leq \min\{n,m\}- d_{\rk}(\C)+1$ (dual distance).   
     \end{enumerate}  
\end{theorem}

\subsection{\emph{q}-Systems and linear sets}\label{subsection:qsystems}

There is a classical way to associate a set of points in $\mathcal{P}\subseteq \PG(k-1, q^m)$ to a projective code in the Hamming metric. The idea is simply to take representatives in $\F_{q^m}^k$ of the points of $\mathcal{P}$ and put them as columns of a $k\times |\mathcal{P}|$ generator matrix $G$ over $\F_{q^m}$ of a code. As in the rank-metric case, such codes depend on the ordering of the points and on their chosen vector representatives, but different choices yield equivalent codes. 
We will call any code in this equivalence class a projective code \emph{associated} with $\mP$ and we will denote it by $\C_\mP$. The same can be done for multisets of points, in which case we arrive at non-projective codes, but we will not consider these in this work. This geometric vision of codes leads to many interesting connections between objects in finite geometry and properties of linear codes. In particular, the Hamming metric can be read from this set of points: for any $u\in \F_{q^m}^{k}$, we have that: 
\begin{equation}\label{eq:hm}
\wt_{H}(u^TG)=n-|\mP\cap\langle u\rangle_{\F_{q^m}} ^\perp|
\end{equation}

In the rank metric, there is analogous interpretation, which associates $q$-systems to codes. We will now introduce these objects. 

\begin{definition}
An \emph{$[n,k]_{q^m/q}$  system} is an $n$-dimensional $\Fq$-space $\mU\leq \Fm^k$ such that $\langle \mU\rangle_{\Fm}=\Fm^k$. A \emph{generator matrix} for $\mU$ is a $k\times n$ matrix over $\F_{q^m}$ whose columns form an $\F_q$-basis for $\mU$. Two $[n,k]_{q^m/q}$ systems $\mU$ and $\mathcal{V}$ are called \emph{equivalent} if there exists $\phi$ in ${\rm GL}_k(q^m)$ such that $\phi(\mU)=\mathcal{V}$.
\end{definition}

A standard way to obtain $[n,k]_{q^m/q}$  systems is to associate them with nondegenerate rank-metric codes. So, given a nondegenerate rank-metric code $\C$, we may associate it with a system $\mU$ by taking a generator matrix of $\C$ and defining $\mU$ to be the $\F_q$-span of its columns. This clearly depends on the choice of the matrix, but if we change the generator matrix we obtain an equivalent system. We will call any system $\mU$ in this equivalence class a system \emph{associated} with $\C$.
For a more detailed description of this correspondence, which involves also the rank metric, the reader is referred to \cite[\S 3]{alfarano2021linear} and \cite{randrianarisoa}. We just point out one important result which is the $q$-analogue of \eqref{eq:hm}: for any $u\in \F_{q^m}^{1\times k}$,
\[\wt_{\rk}(uG)=n-\dim_{\F_q}(\mU\cap\langle u\rangle_{\F_{q^m}}^\perp).\]
In this paper we will show new connections between rank-metric codes (viewed as covering codes) and $q$-systems.
In order to do so, we will use the definition of a linear set. Such objects were introduced by Lunardon in \cite{lunardon1999normal} in order to construct blocking sets and they are subject of intense research over the last years. An in-depth treatment of linear sets can be found in  \cite{polverino2010linear}. 

\begin{definition}
Let $\mU$ be an $\Fmk$ system. The $\F_q$-\emph{linear set} in $\PG(k-1, q^m)$ of rank $n$ associated with $\mU$ is the set $$L_\mU:=\{\langle u \rangle_{\F_{q^m}} :  u\in \mU\setminus\{0\}\},$$
where $\langle u \rangle_{\F_{q^m}}$ denotes the projective point corresponding to $u$.
\end{definition}

\begin{remark}
The original definition of a linear set does not assume that $\langle \mU\rangle_{\Fm}=\Fm^k$. However, if $\dim_{\Fm}(\langle \mU\rangle_{\Fm})=h<k$, then, up to equivalence, we may assume $\mU\leq \F_{q^m}^h$ with $\langle \mU\rangle_{\Fm}=\Fm^h$, and then study $L_\mU$ in $\PG(h-1,q^m)$. 
\end{remark}

Let $0\neq v\in \F_{q^m}^k$ and $P\in \PG(k-1,\Fm)$ be the projective point associated with $v$. We define the \emph{weight} of $P$ in $L_\mU$ as the integer
\[\wt_{\mU}(P):=\dim_{\Fq}(\mU\cap \langle v\rangle_{\F_{q^m}}).\]

\begin{definition}\label{def:scattered}
A linear set $L_{\mU}$ is \emph{scattered} if  $\wt_{\mU}(P)=1$ for each $P\in L_{\mU}$. 
\end{definition}

Any linear set $L_\mU$ of rank $n$ satisfies 
\begin{equation}\label{eq:linear_set} |L_\mU|\leq \frac{q^n-1}{q-1}. \end{equation}
Clearly a linear set $L_\mU$ is scattered if and only if equality holds in \eqref{eq:linear_set}.
Being a set of set of points in $\PG(k-1,q^m)$, a linear set $L_\mU$ can be associated with an $[|L_\mU|,k]_{q^m}$ code $\C_{L_\mU}$. We may hence associate a (projective) Hamming-metric code with a nondegenerate rank-metric code, as summarized in the following diagram.
\begin{equation}\label{eq:Hass}\begin{array}{rcl}
\C & \longrightarrow & \C_{L_\mU}\\
\downarrow & & \uparrow\\
\mU & \longrightarrow & L_\mU
\end{array}
\end{equation}

\begin{definition}
Let $\C$ be an $\Fmk$ rank-metric code. We call the code $\C_{L_\mU}$ obtained as in \eqref{eq:Hass} the \emph{projective Hamming-metric code associated with} $\C$ .
\end{definition}

\begin{remark}
We highlight the fact that the projective Hamming-metric code associated with a rank-metric code defined above is not, in general, the associated Hamming-metric code described in \cite[\S 4.2]{alfarano2021linear}. 
The two definitions coincide if and only if the underlying linear set is {\em scattered} (see for example \cite[\S 4.1]{alfarano2021linear}). If the linear set is scattered, the Hamming-metric code associated with a representative $\Fmk$ code has length $(q^n-1)/(q-1)$. Otherwise, it is shorter. 
We remark that Hamming-metric codes associated with scattered linear sets have been already considered in \cite{blokhuis2000scattered,napolitano2023codes,napolitano2023two,zini2021scattered}.
\end{remark}

\section{Rank-saturating systems}\label{sec:ranksat}

In this section we will introduce the main object of the paper. We will study its properties and relations with covering codes in the rank metric.

Let us start with the notion of a saturating set. 

\begin{definition}
Let $\mS \subseteq  \mathrm{PG}(k-1, q^m)$.
\begin{itemize}
    \item[{\rm (a)}]
      A point $Q \in \mathrm{PG}(k-1, q^m)$ is said to be $\rho$-{\em saturated} by $\mS$ if there exist $\rho+1$ points  $P_1,\ldots,P_{\rho+1}\in \mathcal{S}$ such that $Q\in \langle P_1,\ldots,P_{\rho+1}\rangle_{\F_{q^m}}$. We also say that $\mS$ $\rho$-{\em saturates} $Q$.
    \item[{\rm (b)}]
      The set $\mS$ is called $\rho$-\emph{saturating} set of $\mathrm{PG}(k-1, q^m)$ if every point $Q \in \mathrm{PG}(k-1, q^m)$ is $\rho$-saturated by $\mS$ and $\rho$ is the smallest value with this property.  
\end{itemize}
\end{definition}

 It is well-known (see, for example \cite[Theorem 11.1.2]{huffpless})
 that an $[n,n-k]_{q^m}$ code has (Hamming) covering radius $\rho$ if every element of $\F_{q^m}^k$ is a linear combination of $\rho$ columns of a generator matrix of the dual code and $\rho$ is the smallest value with such a  property. The correspondence between projective systems and linear codes leads to a correspondence between  $(\rho-1)$-saturating sets of size $n$ in $\mathrm{PG}(k-1,q^m)$ and the duals of  $[n,n-k]_{q^m}$ codes of covering radius $\rho$.
In defining the $q$-analogue of such saturating sets, we arrive at a $q$-analogue of this correspondence. 

\begin{definition}\label{def: saturating}
An $\Fmk$ system $\mU$ is \emph{rank-$\rho$-saturating} if $L_\mU$ is a $(\rho-1)$-saturating set in ${\rm PG}(k-1,q^m)$. We call such a linear set a $(\rho-1)$-\emph{saturating linear set}.
\end{definition}

The property of being rank-$\rho$-saturating is clearly invariant under equivalence of $q$-systems. 
The following result offers a characterization of rank-saturating systems which we will use extensively in the remainder of this paper.

\begin{theorem}\label{thm:characterization}
Let $\mU$ be an $\Fmk$ system and let $\{u_1,\ldots,u_n\}$ be an $\F_q$-basis of $\mU$. The following are equivalent:
\begin{itemize}
    \item[{\rm (a)}] $\mathcal{U}$ is rank-$\rho$-saturating.
    \item[{\rm (b)}] For each vector $v\in \F_{q^m}^k$ there exists $\lambda\in \F_{q^m}^{1\times n}$ with ${\rm wt}_{\rk}(\lambda)\le \rho$ such that $$v=\lambda_1 u_1+\ldots+\lambda_n u_n,$$ and $\rho$ is the smallest value with this property.
    \item[{\rm (c)}]
    The full space can be expressed as:
    \[\F_{q^m}^k=\bigcup_{\mathcal{S} :\:\mathcal{S} 
    \leq_{\F_q}\mathcal{U}:\:\dim_{\F_q}\mathcal{S}\leq \rho}
\langle \mathcal{S} \rangle_{\F_{q^m}},
\]
and $\rho$ is the smallest integer with this property.
\end{itemize}
\end{theorem}

\begin{proof}
${\rm (a)}\Rightarrow {\rm (b)}$: Let $0\neq v\in \F_{q^m}^k$ and  $Q=\langle v\rangle_{\F_{q^m}} \in {\rm PG}(k-1,q^m)$. Since $\mU$ is rank-$\rho$-saturating, there exists $\rho$ points $P_1=\langle w_1\rangle_{\F_{q^m}},\ldots, P_\rho=\langle w_\rho\rangle_{\F_{q^m}}$ such that
$$v=\gamma_1w_1+\ldots \gamma_\rho w_\rho$$
with $\gamma_i\in \F_{q^m}$. Now $w_1,\ldots,w_\rho$ are in $L_\mU$, so that, if $u_1,\ldots, u_n$ is an $\F_q$-basis of $\mU$, we have
$$v=\gamma_1(\mu_{1,1}u_1+\ldots+\mu_{1,n}u_n)+\ldots+\gamma_\rho(\mu_{\rho,1}u_1+\ldots+\mu_{\rho,n}u_n),$$
with $\mu_{i,j}\in \F_q$ for all $i,j$. We reorder the terms to obtain:
$$v=\underbrace{(\gamma_1\mu_{1,1}+\ldots+\gamma_\rho\mu_{\rho,1})}_{\lambda_1}u_1+\ldots+\underbrace{(\gamma_1\mu_{1,n}+\ldots+\gamma_\rho\mu_{\rho,n})}_{\lambda_n}u_n.$$
Now, call $\gamma=(\gamma_1,\ldots,\gamma_\rho)\in\F_{q^m}^{1\times \rho}$, $M=(\mu_{i,j})\in \F_q^{\rho\times n}$, and $\lambda=\F_{q^m}^{1\times n}$. We have
$$\lambda=\gamma M,$$
and so $\wt_{\rk}(\lambda)\leq \rho$ (since the rank of $M$ is at most $\rho$). 

\medskip

\noindent ${\rm (b)}\Rightarrow {\rm (c)}$: From (b), any $v\in \F_{q^m}^k$ can be expressed as the linear combination:
\begin{equation}\label{eq:combi}v=\lambda_1u_1+\ldots+\lambda_nu_n
\end{equation}
with $\dim_{\F_q} \langle \lambda_1,\ldots,\lambda_n\rangle_{\F_{q}}\leq \rho.$ Let $\mS=\langle \lambda_1,\ldots,\lambda_n\rangle_{\F_{q}}$. By \eqref{eq:combi}, $v\in \langle \mS\rangle_{\F_{q^m}}$.

\medskip

\noindent ${\rm (c)}\Rightarrow {\rm (a)}$: Take $Q=\langle v\rangle_{\F_{q^m}}\in \PG(k-1,q^m)$. There exists $\mS$, an $\F_q$-subspace of $\mU$ with $\dim_{\F_q}\mS\leq \rho$, such that $v\in \langle\mS\rangle_{\F_{q^m}}$. Let $\{w_1,\ldots,w_\rho\}$ be a set containing a basis of $\mS$ over $\F_q$ and let $P_1,\ldots,P_\rho$ be their corresponding projective points, so that $\langle w_i \rangle_{\Fm}$. These clearly belong to $L_\mU$. Since $v\in \langle\mS\rangle_{\F_{q^m}}$, $Q\in \langle P_1,\ldots,P_\rho\rangle_{\F_{q^m}}$. 

\end{proof}

\begin{remark}
Note that {\rm (b)} does not depend on the choice of the $\F_q$-basis of $\mU$. Indeed, consider two $\F_q$-bases  $\mathcal{B}=\{u_1,\dots,u_n\}$ and $\mathcal{B}^\prime=\{u_1^\prime,\dots,u_n^\prime\}$ of $\mU$. Then, for each $i\in[n]$ we have that $u_i=\sum_{j=1}^n a_ju_j^\prime$ with $a_j\in\F_q$ for $j\in[n]$.  Therefore, 
\[
\sum_{i=1}^n \lambda_i u_i=\sum_{i=1}^n \lambda_i \left(\sum_{j=1}^n a_j u_j^\prime\right)=\sum_{j=1}^n\left(\sum_{i=1}^n \lambda_ia_j\right)u_j^\prime=\sum_{i=1}^n\lambda^\prime_iu_i^\prime, 
\]
which that implies $\wt_{\rk}(\lambda)=\wt_{\rk}(\lambda^\prime)$.
\end{remark}

The following theorem shows that in analogy with the Hamming-metric case, there is a correspondence between rank-saturating systems and rank-metric covering codes, thus encouraging 


\begin{theorem}\label{thm:covering}
Let $\mathcal{U}$ be an $\Fmk$ system associated with a code $\C$. The following are equivalent.
\begin{itemize}
    \item[{\rm (a)}] $\mU$ is rank-$\rho$-saturating.
    \item[{\rm (b)}] $\rho_{\rk}(\C^\perp)=\rho$.
\end{itemize}
\end{theorem}

\begin{proof}
${\rm (a)}\Rightarrow {\rm (b)}$  Let $w\in \F_{q^m}^{1\times n}$, $G$ be a generator matrix for $\C$ and $v=Gw^T\in \F_{q^m}^k$. Since $\mU$ is rank-$\rho$-saturating, by condition {\rm (b)} of Theorem \ref{thm:characterization}, there exists $\lambda\in \F_{q^m}^{1\times n}$ with $\wt_{\rk}(\lambda)\leq \rho$ such that $v=G\lambda^T$. Then $G(w^T-\lambda^T)$, so that $w-\lambda\in \C^\perp$. Since $\rho$ is the least integer with this property, we may conclude that $\rho_{\rk}(\C^\perp)=\rho$.

\noindent ${\rm (b)}\Rightarrow {\rm (a)}$ Let $v\in \F_{q^m}^k$ and $G$ be a generator matrix for $\C$. Let $z$ any vector in $\F_{q^m}^{1\times n}$ such that $v=Gz^T$. By the definition of rank covering radius, there exists $w\in \C^\perp$ (i.e. satisfying $Gw^T=0$) such that $\wt_{\rk}(z-w)\leq \rho$. Call $\lambda=z-w$. We have $v=Gz^T=G(z^T-w^T)=G\lambda^T$. Since $\rho$ is the least integer with this property, we may conclude that $\mU$ is rank-$\rho$-saturating.
\end{proof}

\begin{corollary}
Let $\C$ be an $\Fmk$ rank-metric code and let $\mU$ be an $\Fmk$ system associated with $\C$. Then
$$\rho_{\rk}(\C^\perp)=\rho_H((\C_{L_\mU})^\perp),$$
where $\C_{L_\mU}$ is the projective Hamming-metric code associated with $\C$.
\end{corollary}

\begin{proof}
This follows immediately by Theorem \ref{thm:covering} and by the definition of rank-$\rho$-saturating system.
\end{proof} 

We close this section by reformulating some known results (see \cite{byrne2017covering}) on the rank-covering radius, in the language of saturating systems.

\begin{corollary}
Let $\mathcal{U}$ be a rank-$\rho$-saturating $\Fmk$ system associated with a code $\C$. Then
\[\rho\leq s_\rk(\C) \quad \text{ and } \quad \rho\leq \min\{m,n\}- d_{\rk}(\C)+1.\]
\end{corollary}

\begin{proof}
These are direct consequences of Theorem \ref{thm:edb}.
\end{proof}

\begin{corollary}\label{cor:gab}
     Let $\C$ be an $[n,k]_{q^m/q}$ generalized Gabidulin code and let $\mU$ be an $\Fmk$ system associated with $\C$. Then 
     $\mU$ is a rank-$k$-saturating system.
\end{corollary}
\begin{proof}
    The statement follows immediately from the fact that $\rho_{\rk}(\C^\perp) = k$. 
\end{proof}

\section{Bounds on the dimension of rank-saturating systems}\label{sec:bounds}

The classical covering problem, as presented for example in \cite{coveringcodes}, is as follows: given $n$ and $\rho$, estimate the least number of spheres of radius $\rho$ (with respect to the metric considered) such that every vector in the ambient vector space of dimension $n$ is contained in their union, i.e. such that the union of the balls of radius $\rho$ {\em covers} this $n$-dimensional space.
In the framework of linear codes, this is equivalent
to asking how large the rate of a code (that is the ratio between the dimension of the code and $n$) must be
in order to obtain a covering of the ambient space by balls centred at codewords.
In terms of rank-$\rho$-saturating systems, by Theorem \ref{thm:covering} one may ask to find the least value of $n$ such that an $[n,k]_{q^m/q}$ rank-$\rho$-saturating system exists, for fixed $k$ and $\rho$.

\begin{definition}
We denote by $s_{q^m/q}(k,\rho)$ the minimal $\F_q$-dimension of any rank-$\rho$-saturating system in $\F_{q^m}^k$. That is,
\[
s_{q^m/q}(k,\rho):= \min\{n :\exists \text{ an } [n,k]_{q^m/q}  \text{ rank-}\rho\text{-saturating system}\}.
\]
\end{definition}

The rest of this paper is devoted to obtaining bounds on this quantity: we will first give a lower bound and then provide upper bounds arising from explicit constructions of rank-$\rho$-saturating systems.
We will use the following result.

\begin{lemma}[\!\!{\cite[Corollary 2.3]{gadouleauphd}}]\label{lem:gadconst}
    Let $a,b$ be positive integers, with $b\leq a$. Then
    \[\left[\begin{array}{c}
	         a\\
	         b
	    \end{array}\right]_q <  \frac{q^{b(a-b)}}{(1/q)_{\infty}},
	\]
	where $\displaystyle (1/q)_{\infty}:=\prod_{i=1}^{\infty}(1-q^{-i}).$
\end{lemma}

The following has been obtained with a slightly different approach in \cite[Proposition 14]{gadouleau2008packing}.

\begin{theorem}
\label{thm_LowerBound}
		Let $\mathcal{U}$ be a rank-$\rho$-saturating $\Fmk$ system.
		Then
	\[
	    \left[\begin{array}{c}
	         n\\
	         \rho
	    \end{array} \right]_{q}\geq q^{m(k-\rho)}.
	\]
In particular, we have the following:
\begin{equation}\label{eq:fullineq}
n \geq \begin{dcases}
           \left\lceil \frac{mk}{\rho}\right\rceil - m + \rho & \text{ if } q>2,\\
           \left\lceil \frac{mk-1}{\rho}\right\rceil - m + \rho & \text{ if } q=2, \rho> 1,\\
           m(k-1)+1 & \text{ if } q=2,\rho =1.
\end{dcases}
\end{equation}
\end{theorem}
\begin{proof}
	Let us consider the set $\Pi_{\rho}$ of all $\F_{q^m}$-subspaces spanned by $\rho$ $\F_q$-linearly independent elements of $\mU$; since the $\F_q$-dimension of these subspaces is $\rho$, the rank of the $\Fq$-span of their coefficients is at most $\rho$.
	As $\mathcal{U}$ saturates $\F_{q^m}^k$, from Theorem \ref{thm:covering}, we know that $\Pi_{\rho}$ must cover the latter, i.e. that $\F_{q^m}^k = \bigcup_{V \in \Pi_{\rho}} V$. Therefore,
\begin{equation}\label{eq:ineq}
 	    \left[\begin{array}{c}
	         n\\
	         \rho
	    \end{array} \right]_{q}\cdot q^{m\rho}\geq q^{mk}.
\end{equation}
\noindent If $q=2$ and $\rho=1$, then from \eqref{eq:ineq} we get that $2^n-1\geq 2^{m(k-1)}$ and hence $n\geq m(k-1)+1$.
 From Lemma \ref{lem:gadconst},
\[
    \left[\begin{array}{c}
	         a\\
	         b
	    \end{array}\right]_q < (1/q)_{\infty}^{-1}\cdot q^{b(a-b)}, \mbox{for } a,b \in \mathbb N.\\
\]	
So
\[
	(1/q)_{\infty}^{-1}\cdot q^{\rho(n-\rho)}\cdot q^{m\rho}>q^{mk}.
\]
Hence
\[n \geq  \left\lceil\frac{m}{\rho}(k-\rho) + \rho - \frac{\lfloor \log_q((1/q)_\infty^{-1}) \rfloor}{\rho}\right\rceil= \left\lceil \frac{mk-\lfloor \log_q((1/q)_\infty^{-1}) \rfloor}{\rho}\right\rceil -m+\rho.\]
The result now follows since $(1/q)_{\infty}^{-1}< q$ for all $q>2$, and is strictly less than $4$ for $q=2$.
\end{proof}

By Theorem \ref{thm_LowerBound}, we obtain an immediate lower bound:
\begin{align}\label{eq:lb}
s_{q^m/q}(k,\rho) \geq \begin{dcases}
           \left\lceil \frac{mk}{\rho}\right\rceil - m + \rho & \text{ if } q>2,\\
           \left\lceil \frac{mk-1}{\rho}\right\rceil - m + \rho & \text{ if } q=2, \rho > 1,\\
           m(k-1)+1 & \text{ if } q=2,\rho =1.
       \end{dcases}
\end{align}
Note that in the case $\rho=1$, the bound of (\ref{eq:lb}) is attained, i.e.,
$s_{q^m/q}(k,1)=m(k-1)+1$.\\
To see this, let $v\in \F_{q^m}^k$, $v\neq 0$ and let $v'\notin \langle v\rangle_{\F_{q^m}}^\perp $. Consider the $[m(k-1)+1,k]_{q^m/q}$ system:
\[\mU=\langle v'\rangle_{\F_q}+\langle v\rangle_{\F_{q^m}}^\perp,\]
which is clearly a rank-$1$-saturating system, because $L_\mU=\PG(k-1,q^m)$. Let $\C$ be the code whose generator matrix has the elements of an $\F_q$-basis of the system $\mU$ as its columns. The dual code $C^\perp$ is an $[m(k-1)+1,m(k-1)+1-k]_{q^m/q}$ with rank covering radius $1$ and it is the shortest code with this property for this dimension and $m$.

We now obtain upper bounds on $s_{q^m/q}(k,\rho)$. 
To start with, we give a generalization of the previous construction. 

\begin{theorem}\label{thm:generalUB}
Any $[m(k-\rho)+\rho,k]_{q^m/q}$ system $\mU$ with generator matrix
\[
G:=\left[ \begin{array}{c|c}
I_{\rho} & \mathbf{0}   \\
\hline
\mathbf{0} & G' \end{array}\right],
\]
is rank-$\rho$-saturating. In particular, 
\[s_{q^m/q}(k,\rho)\leq m(k-\rho)+\rho.\]
\end{theorem}

\begin{proof}
Let $\F_{q^m}=\F_q[\alpha]$. By \cite[Proposition 3.16.]{alfarano2021linear}, we have that, up to equivalence, the system $\mU$ has generator matrix
\[G:=\left[ \begin{array}{c|c|c|c|c}
I_{\rho} & \mathbf{0}   & \mathbf{0}     & \cdots     & \mathbf{0}\\
\hline
\mathbf{0} & I_{k-\rho} & \alpha I_{k-\rho} & \cdots & \alpha^{m-1}I_{k-\rho} \end{array}\right],\]
while $\mU$ itself is given by:
\[\mathcal{U}=\left\{\left(\begin{array}{c}u\\\hline \omega \end{array}\right) : u\in\F_{q}^\rho,\omega\in\F_{q^m}^{k-\rho}\right\}.\]
Let $v\in \F_{q^m}^k$ and suppose that $v_{i_1},\ldots, v_{i_r}\neq 0$, for some $i_j\in [\rho]$. Then $v$ can be expressed as:
\[v=v_{i_1}\left(\begin{array}{c}e_{i_1}\\\hline \frac{\omega_1}{v_{i_1}} \\\vdots\\ \frac{\omega_{k-\rho}}{v_{i_1}} \end{array}\right)+v_{i_2}\left(\begin{array}{c}e_{i_2}\\\hline \frac{\omega_1}{v_{i_2}} \\\vdots\\ \frac{\omega_{k-\rho}}{v_{i_2}} \end{array}\right)+\ldots+v_{i_r}\left(\begin{array}{c}e_{i_r}\\\hline \frac{\omega_1}{v_{i_r}} \\\vdots\\ \frac{\omega_{k-\rho}}{v_{i_r}} \end{array}\right),\]
where $e_1,\ldots,e_\rho$ is the standard basis of $\F_{q}^\rho$.
Clearly, each of these $r\leq\rho$ vectors belongs to $\mU$. 
Any vector whose first $\rho$ coordinates are non-zero requires exactly $\rho$ vectors in $\mU$ and hence the system is rank-$\rho$-saturating.
\end{proof}

Since we have equality between the lower and the upper bound for $\rho=1$ and for $\rho=k$, the bound of (\ref{eq:lb}) is attained in these cases. 

We now study some properties of the function $s_{q^m/q}(k,\rho)$.

\begin{lemma}\label{lemma:ineq}
Let $\mU$ be a rank-$\rho$-saturating $[n,k]_{q^m/q}$ system.
The following are equivalent.
\begin{enumerate}
    \item 
      $L_\mU$ is not scattered.
    \item
      $\mU$ has an $\Fq$-basis $\{u_1,\ldots,u_n\}\subseteq \F_{q^m}^k$ with the property that
\[
u_n = \lambda \sum_{j=1}^{n-1}l_{\rho+1,j} u_j,
\] 
for some $l_{\rho+1,j}\in \Fq,1\leq j\leq n-1$ and $\lambda \in \F_{q^m}\setminus \F_q$.
\end{enumerate}
If either of the above equivalent properties hold, then
$\mU$ contains a rank-$\rho'$-saturating \\$[n-1,k]_{q^m/q}$ system satisfying $\rho'\leq \rho+1$.
In particular, one such system is given by $\langle u_1,\ldots,u_{n-1}\rangle_{\Fq}$.
\end{lemma}
\begin{proof}
The equivalence of the two statements given above is clear: $L_\mU$ is scattered if and only no two members of $\mU$ are $\F_{q^m}$-multiples of the same vector in $\F_{q^m}^k$.
For any vector $v\in \F_{q^m}^k$, 
\[v=\sum_{i=1}^\rho\lambda_i\sum_{j=1}^nl_{i,j}u_j\]
for some $\lambda_i\in \F_{q^m}$ and $l_{i,j}\in \F_q$. Therefore,
\begin{align*}
  v & = \sum_{i=1}^\rho\lambda_i\sum_{j=1}^{n-1}l_{i,j}u_j+\sum_{i=1}^\rho\lambda_il_{i,n}u_n\\
  & = \sum_{i=1}^\rho\lambda_i\sum_{j=1}^{n-1}l_{i,j}u_j+\sum_{i=1}^\rho\lambda_il_{i,n} \lambda \sum_{j=1}^{n-1}l_{\rho+1,j} u_j\\
  & = \sum_{i=1}^{\rho+1}\lambda_i\sum_{j=1}^{n-1}l_{i,j}u_j,
\end{align*}
where $\lambda_{\rho+1}=\sum_{i=1}^\rho\lambda_il_{i,n} \lambda\in \F_{q^m}$.
\end{proof}


Using similar arguments as in the classical Hamming-metric case (see \cite[\S 11.5]{huffpless}), we have the following results.    

\begin{theorem}[Monotonicity]\label{th:ineq}
The following hold:
\begin{itemize}
    \item[{\rm (a)}] If $\rho<\min\{k,m\}$, then $s_{q^m/q}(k,\rho+1) \leq s_{q^m/q}(k,\rho).$
    \item[{\rm (b)}] $s_{q^m/q}(k,\rho) \leq s_{q^m/q}(k+1,\rho)-1$.
    \item[{\rm (c)}] If $\rho<m$, then $s_{q^m/q}(k+1,\rho+1) \leq s_{q^m/q}(k,\rho)+1$.
\end{itemize}
\end{theorem}

\begin{proof}
\noindent {\rm (a)} 
    Let $n>k$ and let $n=s_{q^m/q}(k,\rho)$. Let $G\in \F_{q^m}^{k \times n}$ be a
    generator matrix associated with a rank-$\rho$-saturating $[n,k]_{q^m/q}$ system $\mU$.
    We may assume that $G=[I_k|u_{k+1}\ldots,u_{n-1},y]$ for some $y,u_i \in \mU$. 
    Assume further, that over all such choices of $\mU$ and $G$, that $y$
    has minimal rank weight. 
    
    If $\wt_{\rk}(y^T)=1$ then $\mU$ satisfies the hypothesis of Lemma \ref{lemma:ineq} and so there exists a $(\rho+1)$-rank-saturating system of length $n-1$.
    We thus assume that $\wt_{\rk}(y^T)=\ell \geq 2$. 
    Let $\{b_1:=y_{i_1},\dots,b_{\ell}:=y_{i_\ell}\}$ be an $\Fq$-basis
    of $\langle y_1,\dots,y_k\rangle_{\F_q}$.
    We have that $y=b_\ell \sum_{j=1}^k p_j e_j +y'$ for some $p_j \in \Fq$ and $y'\in \F_{q^m}^k$ satisfying
    $y' = (\sum_{j=1}^{\ell -1} a_{1j} b_j,\dots,\sum_{j=1}^{\ell -1}a_{kj} b_j)^T$ for some $a_{ij} \in \Fq$.
    Consider the matrix $G'=[I_k|u_{k+1}\ldots,u_{n-1},y']$,
    and the corresponding $\rho'$-rank-saturating $\Fmk$ system $\mU'$ spanned by its columns. 
    Let $w \in \F_{q^m}^k$. There exists $z \in \F_{q^m}^n$ of rank at most $\rho$ such that $w=Gz$. Therefore,
    \begin{align*}
        w & = Gz \\
          & = \sum_{i\in [k]} z_i e_i +\sum_{i=k+1}^{n-1} z_iu_i + z_n(b_\ell \sum_{i \in [k]}p_ie_i+y')\\
          & =  \sum_{i\in [k] } (z_i +z_nb_\ell p_i)e_i 
          +\sum_{i=k+1}^{n-1} z_iu_i +z_n y'\\
          & =G'(z+z_nb_\ell \sum_{i \in [k]}p_ie_i).
    \end{align*}
    Let $z' = z+z_nb_\ell \sum_{i \in [k]}p_ie_i$. Clearly, $\wt_{\rk}((z')^T) \leq \wt_{\rk}(z^T)+1\leq \rho+1$ and so $\rho' \leq \rho+1$.
    If $\rho'=\rho+1$, then we have $s_{q^m/q}(k,\rho+1)\leq n = s_{q^m/q}(k,\rho)$ and hence the statement of the theorem will follow.
    
    If $y'\in \Fq^k$, then $\mU'$ is an $[n-1,k]_{q^m/q}$ system and so $\rho' = \rho+1$. 
    Suppose then that $x\notin \Fq^k$.
    Since $\wt_{\rk}((y')^T)<\wt_{\rk}(y^T)$, by our choice of
    $\mU$ and $G$, it must be the case that $\rho'\neq \rho$.
    Suppose now that $\rho'\leq \rho-1$. If $\wt_{\rk}((y')^T)=1$ then $\mU'$ satisfies the hypothesis of Lemma~\ref{lemma:ineq} and so there exists a rank-$\rho''$-saturating $[n-1,k]_{q^m/q}$ system $\mU''$ with $\rho'' \leq \rho$, yielding a contradiction to the fact that $n=s_{q^m/q}(k,\rho)$.
    We hence assume that $\wt_{\rk}((y')^T)\geq 2$. Apply a similar argument as before to produce a matrix
    $G'=[I_k|u_{k+1}\ldots,u_{n-1},y'']$ with associated rank-$\rho''$-saturating system $\mU''$
    satisfying $\rho'' \leq \rho'+1 \leq \rho$ and $\wt_{\rk}((y'')^T) < \wt_{\rk}(y^T)$. Again, by our choice of $G$ and $\mU$, it must be the case that $\rho''\leq \rho-1$. Continue, iterating the same argument to produce a sequence of generator matrices $G^{(i)}=[I_k|u_{k+1},\ldots ,u_{n-1},y^{(i)}]$ 
    and associated $[n-1,k]_{q^m/q}$ rank-$\rho^{(i)}$-saturating systems $\mU^{(i)}$
    with $\wt_{\rk}((y^{(i)})^T)<\wt_{\rk}((y^{(i-1)})^T)$ at each step. This sequence will terminate at some $r$
    for which $\wt_{\rk}((y^{(r)})^T)=1$, in which case we may apply Lemma \ref{lemma:ineq} to arrive at a contradiction. We deduce that $\rho' =\rho+1$ and so the result follows.

\medskip

\noindent {\rm (b)}  Let $n>k$ and let $n=s_{q^m/q}(k,\rho)$. Let $G=[\:I_{k+1}|\:A\:]\in \F_{q^m}^{(k+1) \times n}$ be a generator matrix of a rank-$\rho$-saturating $[n,k+1]_{q^m/q}$ system $\mU$. Consider the matrix $G'=[\:I_{k}|\:A'\:]\in \F_{q^m}^{k \times (n-1)}$ found by deleting the first column and row of $G$.
   Let $w' \in \F_{q^m}^k$ and let $w=(0,w')^T \in \F_{q^m}^{k+1}$.
   Since $\mU$ is rank-$\rho$-saturating, there exists $z \in \F_{q^m}^n$ of rank at most $\rho$ such that $w=Gz$ and so $w'=G'z'$, where $z'=(z_2,\ldots,z_n)^T$.
   Since $\wt_{\rk}((z')^T)\leq \wt_{\rk}(z^T)\leq \rho$, then $G'$ generates 
   an $[n-1,k]_{q^m/q}$ rank-$\rho'$-saturating system $\mU'$ with $\rho'\leq \rho$. Therefore, by {\rm (a)},
   \[s_{q^m/q}(k,\rho)\leq s_{q^m/q}(k,\rho') \leq n-1 = s_{q^m/q}(k+1,\rho)-1.\]

\medskip

\noindent {\rm (c)}   Let $n=s_{q^m/q}(k,\rho)$. Let $G\in \F_{q^m}^{k \times n}$ be a generator matrix of a rank-$\rho$-saturating $[n,k]_{q^m/q}$ system $\mU$.  
  Consider the matrix 
  \[
  G' =\left[ \begin{array}{cc}
G & 0 \\
0 & 1 \end{array}\right] \in \F_{q^m}^{(k+1)\times (n+1)},
  \]
  which generates a rank-$\rho'$-saturating $[n+1,k+1]_{q^m/q}$ system $\mU'$.
  It is straightforward to check that for any $w \in \F_{q^m}^{k+1}$, there exists
  $z \in \F_{q^m}^{n+1}$ of rank at most $\rho+1$ such that $w=G'z$.
  Again by {\rm (a)}, we have 
  \[s_{q^m/q}(k+1,\rho+1) \leq s_{q^m/q}(k+1,\rho')\leq n+1 \leq s_{q^m/q}(k,\rho)+1.\]
\end{proof}    

In the following, we define the direct sum of systems to obtain recursive bounds, in analogy with \cite{davydov2000saturating,ughi1987saturated}. 

\begin{definition}
For each $i\in \{1,2\}$, let $\mU_i$ be an $[n_i,k_i]_{q^m/q}$ system, associated with an $[n_i,k_i]_{q^m/q}$ code $\mC_i$.
Let $f:\F_{q^m}^{1\times n_1} \longrightarrow \F_{q^m}^{1\times n_2}$ be an $\F_{q^m}$-linear map.
    The code 
    \[\mC:=\{(u,f(u)+v) : u \in \C_1, v \in \mC_2\}\] 
    is an $[n_1+n_2,k_1+k_2]_{q^m/q}$, which we call the $f$-sum of $\mC_1$ and $\mC_2$. Its associated $[n_1+n_2,k_1+k_2]_{q^m/q}$ system is called the $f$-sum of $\mU_1$ and $\mU_2$, which we denote by $\mU_1 \oplus_{f} \mU_2$.
\begin{enumerate}
\item    
    If $f$ is the identity map, the $f$-sum of $\mU_1$ and $\mU_2$ is called the \emph{Plotkin sum} of $\mU_1$ and $\mU_2$.
\item   
    If $f$ is the zero map, the $f$-sum of $\mU_1$ and $\mU_2$ is called the  \emph{direct sum} of $\mU_1$ and $\mU_2$, which we denote by $\mU\oplus \mU^\prime$.
\end{enumerate}                 
\end{definition}

\begin{theorem}\label{prop:DirectSum}
For each $i\in \{1,2\}$, let $\mU_i$ be an $[n_i,k_i]_{q^m/q}$ 
rank-$\rho_i$-saturating system, associated with an $[n_i,k_i]_{q^m/q}$ code $\mC_i$.
Let $f:\F_{q^m}^{1\times n_1} \longrightarrow \F_{q^m}^{1\times n_2}$ be an $\F_{q^m}$-linear map.
Then $\mU_1 \oplus_{f} \mU_2$ is an $[n_1+n_2,k_1+k_2]_{q^m/q}$ system 
that is rank-$\rho$-saturating, where 
  $\rho\leq \rho_1+\rho_2.$
In particular, if $\rho_1+\rho_2\leq \min\{k_1+k_2,m\}$, then
\[s_{q^m/q}(k_1+k_2,\rho_1+\rho_2) \leq s_{q^m/q}(k_1,\rho_1)+s_{q^m/q}(k_2,\rho_2).\]
\end{theorem}
\begin{proof}  
$\C=\C_1 \oplus_f \C_2$ has a generator matrix of the form 
    \[
       G = \left[
           \begin{array}{cc}
              G_1 & G' \\
              0   & G_2
           \end{array}
           \right],
    \]
    where $G_i$ is a generator matrix for $\C_i$ for each $i$ and $G' \in \F_{q^m}^{k_1 \times n_2}$. Let $\mU_i$ be the system generated by $G_i$ and let $\mU'$ be the system generated by $G'$.
    Since $\mU_1$ is rank-$\rho_1$-saturating, $\mU_1+\mU'$ is $\rho'$-saturating for some $\rho' \leq \rho_1$.
Let $v\in \F_{q^m}^{k_1+k_2}$ and write $v=(v^{(1)},v^{(2)})^T$ with each $v^{(i)} \in \F^{k_i}_{q^m}$. 
There exists $(\lambda^{(1)},\lambda')^T \in \F_{q^m}^{n_1+n_2}$ of rank weight at most $\rho'$ and $\lambda^{(2)} \in \F_{q^m}^{n_2}$ of rank weight at most $\rho_2$ such that:
 \[
      v= \left[
           \begin{array}{c}
              v^{(1)} \\ 
              v^{(2)}
           \end{array}
           \right] =
           \left[
           \begin{array}{cc}
              G_1 & G' \\
              0   & G_2
           \end{array}
           \right] 
           \left[
           \begin{array}{c}
              \lambda^{(1)} \\ 
              \lambda'+\lambda^{(2)}
           \end{array}
           \right],
    \]
    and clearly $\lambda =(\lambda^{(1)},\lambda'+\lambda^{(2)})^T$ has rank weight at most $\rho_1+\rho_2$.

Suppose now that $\mU_i$ has $\F_q$-dimension $s_{q^m/q}(k_i,\rho_i)$ for $i \in \{1,2\}$. Then $\mU_1\oplus_f \mU_2$ has $\F_q$-dimension $s_{q^m/q}(k_1,\rho)+s_{q^m/q}(k_2,\rho_2)$. Since $\mU_1\oplus_f \mU_2$ is rank- $\rho''$-saturating with $\rho''\leq \rho_1+\rho_2$, by Theorem \ref{th:ineq}, 
\[s_{q^m/q}(k_1+k_2,\rho_1+\rho_2)\leq s_{q^m/q}(k_1+k_2,\rho'')\leq s_{q^m/q}(k_1,\rho_1)+s_{q^m/q}(k_2,\rho_2),\] 
if $\rho_1+\rho_2\leq \min\{k_1+k_2,m\}$.
\end{proof}

\begin{remark}
The direct sum $\mU_1 \oplus \mU_2$ may be $\rho$-rank-saturating with $\rho < \rho_1 +\rho_2$, as the following example shows. Let $\F_{16}=\F_2[\alpha]$ with $\alpha^4=\alpha+1$. Let $\mU_1$ be the $[2,2]_{16/2}$ system and $\mU_2$ be the  $[3,1]_{16/2}$ system defined, respectively, by
\[G_1=\left[ \begin{array}{cc}
1 & 0 \\
0 & 1 \end{array}\right]\quad \text{ and }\quad  G_2=\left[ \begin{array}{ccc}
1 & \alpha   & \alpha^5  \end{array}\right].\]
The system $\mU_1$ is rank-$2$-saturating and the system $\mU_2$ is rank-$1$-saturating, while $\mU_1\oplus\mU_2$ is rank-$2$-saturating (which can be verified directly with {\sc Magma}).
\end{remark}

\begin{corollary}
 $$s_{q^m/q}(tsh,ts)\le t\cdot  s_{q^m/q}(sh,s).$$
\end{corollary}
\begin{proof}
We proceed by induction on $t$ (for $t=1$ it is clear).
By Theorem \ref{prop:DirectSum} and by induction hypothesis, we get 
\begin{align*}
s_{q^m/q}(tsh,ts) & \le  \, s_{q^m/q}((t-1)sh,(t-1)s)+ s_{q^m/q}(sh,s)\\
& \leq (t-1)s_{q^m/q}(sh,s)+s_{q^m/q}(sh,s)=t\cdot s_{q^m/q}(sh,s).
\end{align*}
\end{proof}

\section{Constructions}\label{sec:const}

In this section, we present some geometric constructions of rank-saturating systems of small $\F_q$-dimension, following the lines of \cite{denaux2021constructing,davydov2000saturating,ughi1987saturated,davydov1995constructions}, wherein, as we have already mentioned, the two main approaches involve constructions using cutting blocking sets and mixed subgeometries.

\subsection{Constructions from linear cutting blocking sets}

Let us first introduce the notion of a cutting blocking set.

\begin{definition}
A subset $\mathcal{M} \subseteq \PG(k-1,q)$ is a \emph{cutting blocking set} (or \emph{strong blocking set}) if for every hyperplane $\HHH$ of $\PG(k-1,q)$, we have: 
\[\langle \mathcal{M} \cap \HHH \rangle =\HHH.\]
 \end{definition} 

Such sets were introduced in \cite{1930-5346_2011_1_119}, with the original name of \emph{strong blocking sets}, in connection to $\rho$-saturating sets. More explicitly, we have the following result.

\begin{theorem}[Theorem 3.2. of \cite{1930-5346_2011_1_119}]\label{thm:saturating}
Any cutting blocking set in a subgeometry $\mathrm{PG}(k-1,q)$ of  $\mathrm{PG}(k-1,q^{k-1})$ is a $(k-2)$-saturating set in $\mathrm{PG}(k-1,q^{k-1})$.
\end{theorem}

In \cite{bonini2020minimal}, they were reintroduced, with the name of \emph{cutting blocking sets}, in order to construct a particular family of minimal codes.

\begin{definition}
 An $[n,k]_{q^m}$ code $\C$ is \emph{minimal} if for every $c,c'\in \C$, $\{i:c'_i\neq 0\}\subseteq \{i:c_i\neq 0\}$ implies $c'=\lambda c$ for some $\lambda\in \F_{q^m}$.
\end{definition}

Such codes have been the subject of intense research over the last twenty years. In \cite{alfarano2019geometric,Full_Characterization} it is shown that they are the geometrical counterparts of minimal codes, via the correspondence introduced in Subsection \ref{subsection:qsystems}. One of the main problems in the theory of minimal codes is the construction of families of short-length codes, which is equivalent to constructing small strong blocking sets. Some recent results can be found in \cite{bartoli2021small,heger2021short,alfarano2020three,alon2023strong,alfarano2023outer,bishnoi2023blocking}.

The $q$-analogue of a cutting blocking set is defined as follows.

\begin{definition} 
 A $\Fmk$ system $\mU$ is called a \emph{linear cutting blocking set} if for  every $\Fm$-hyperplane $\mathcal{H}$ we have $\langle \mathcal{H}\cap \mU\rangle_{\Fm}=\mathcal{H}$.
\end{definition}

Linear cutting blocking sets were introduced recently in \cite{alfarano2021linear}, in connection with minimal codes in the rank metric. In order to define these, we introduce the notion of rank-support. Fix an ordered basis $\Gamma=\{\gamma_1,\ldots,\gamma_m\}$ of 
$\F_{q^m}/\F_q$. For a word 
$c \in \F_{q^m}^{1\times n}$, let $\Gamma(c) \in \F_q^{n \times m}$ be the matrix such that
\[c_i= \sum_{j=1}^m \Gamma(c)_{ij} \gamma_j.\]
The \emph{rank-support} of $c$, which we denote by $\sigma^{\rk}(c)$ is the column space of $\Gamma(c)$.

\begin{definition}
 An $[n,k]_{q^m/q}$ code $\C$ is \emph{minimal} if for every $c,c'\in \C$, $\sigma^{\rk}(c')\subseteq \sigma^{\rk}(c)$ implies $c'=\lambda c$ for some $\lambda\in \F_{q^m}$.
\end{definition}

As shown in \cite{alfarano2021linear}, a $q$-system is a linear cutting blocking set if and only if the associated rank-metric code is minimal.
We will show that, as in the classical setting, linear cutting blocking sets give rise to rank-saturating systems. 

\begin{theorem}\label{thm:cutting}
Let $\mU$ be an $\Fmk$ system. If $\mU$ is a linear cutting blocking set, then it is a rank-$(k-1)$-saturating $[n,k]_{q^{m(k-1)}/q}$ system.
\end{theorem}

\begin{proof}
The system $\mU$ is a linear cutting blocking set in $\F_{q^m}^k$, so that the $[n,k]_{q^m/q}$ code $\C$ associated with $\mU$ is a minimal code in the rank metric by \cite[Corollary 5.7]{alfarano2021linear}. Then the projective Hamming-metric code $\C_{L_\mU}$ associated with $\C$ is a minimal code in the Hamming metric by \cite[Theorem~5.13]{alfarano2021linear} (indeed, $\C_{L_\mU}$ is the projectivisation of the code $\C^H$ in that reference, and we are using also the trivial fact that a code is minimal if and only if its projectivisation is minimal). Hence $L_\mU$ is a cutting blocking set in ${\rm PG}(k-1,q^m)$. Then $L_\mU$ is a $(k-2)$-saturating set in ${\rm PG}(k-1,q^{m(k-1)})$ by Theorem \ref{thm:saturating}. By definition, this means that $\mU$ is a rank-$(k-1)$-saturating $[n,k]_{q^{m(k-1)}/q}$ system.
\end{proof}

\begin{corollary}\label{coro:lincutsat}
For every $m,k\geq 2$,
\[k+m-1\leq s_{q^{m(k-1)}/q}(k,k-1)\leq l_{q^m/q}(k)\leq 2k+m-2,\]
where $l_{q^m/q}(k)$ is the minimum ${\mathbb F}_{q}$-dimension of a linear cutting
blocking set in ${\mathbb F}_{q^m}^k$.
\end{corollary}

\begin{proof}
The upper bound is a direct consequence of Theorem \ref{thm:cutting} and of \cite[Corollary 6.11.]{alfarano2021linear}, where it is shown that for every $m,k\geq 2$, there exists a $[2k+m-2,k]_{q^{m}/q}$ linear cutting blocking set. The lower bound is the one by Theorem \ref{thm_LowerBound}.
\end{proof}

\begin{remark}
Quite remarkably, the lower bound coincides with the one for linear cutting blocking set given in \cite[Corollary 5.10]{alfarano2021linear}, calculated over the subfield $\F_{q^m}$. Note however that in \cite{neriMRD} it is proved that the bound is not sharp for linear cutting blocking sets when $m<(k-1)^2$. It would be interesting to know if a similar result holds also for saturating systems.
\end{remark}



\begin{theorem}
The equality 
\[s_{q^{2r}/q}(3,2)=r+2\]
holds if one of the following is true:
\begin{itemize}
    \item[\rm (a)] $r\not\equiv 3,5 \bmod 6$ and $r\ge 4$;
    \item[\rm (b)] $\gcd(r,(q^{2s}-q^s+1)!)=1$, $r$ odd, $1\leq s\leq r$, $\gcd(r,s)=1$;
    \item[\rm (c)] $r=5$, $q=p^{15h+s}$, $p\in\{2,3\}$, $\gcd(s,15)=1$;  
    \item[\rm (d)] $r=5$, $q=5^{15h+1}$;
    \item[\rm (e)] $r=5$, $q$ odd, $q\equiv 2,3\bmod 5$ and for $q=2^{2h+1}$, $h\geq 1$.
\end{itemize}
\end{theorem}

\begin{proof}
According to \cite{alfarano2021linear,lia2023short,bartoli2021evasive}, under any of these hypothesis $[r+2,3]_{q^r/q}$ linear cutting blocking sets exist. So by Theorem \ref{thm:cutting}, rank-$2$-saturating $[r+2,3]_{q^{2r}/q}$ systems exist. The equality comes from the fact that in this case the upper bound meets the lower bound.
\end{proof}

\begin{remark}
Let us remark that, according to \cite{gruica2022generalised},  $[r+3,3]_{q^r/q}$ linear cutting blocking sets exist for any $m$ and $q$. So, in general $s_{q^{2r}/q}(3,2)\in\{r+2,r+3\}$. 
\end{remark}

\begin{example}
Let $\lambda$ in $\F_{16}$ such that $\lambda^4=\lambda+1$. The $[6,3]_{16/2}$ system with generator matrix
 $$G= \begin{pmatrix} \lambda^{4} & \lambda^{10} & \lambda^{8} & \lambda^{3} & \lambda^{9} & \lambda^{7} \\
 \lambda^{14} & \lambda^{8} & \lambda & \lambda^{8} & 0 & \lambda^{8} \\
 \lambda^{10} & 0 & \lambda^{6} & \lambda^{5} & \lambda^{11} & \lambda^{3} \end{pmatrix},$$
is a linear cutting blocking set, as shown in \cite[Example 6.9]{alfarano2021linear}. So the $[6,3]_{256/2}$ system $\mU$ with the same generator matrix is a rank-$2$-saturating system. It has the smallest $\F_2$-dimension. The linear set $L_\mU$ is scattered.
\end{example}

\begin{remark}\label{remark:neri} In \cite{neriMRD} it is shown that, for all $q$, there exists an $[8,4]_{q^3/q}$ linear cutting blocking set. Therefore, by Theorem \ref{thm:cutting} there exists a rank-$3$-saturating  $[8,4]_{q^{9}/q}$ system and hence 
\[6\leq s_{q^9/q}(4,3)\leq 8.\]
In this case, their construction is independent of $q$.

On the other hand, for $q=2^h$ with $h$ odd, they show that the $[8,4]_{q^4/q}$ system
\[\mU=\left\{\left(\begin{array}{c} x\\y\\x^{q}+y^{q^2}\\x^{q^2}+y^q+y^{q^2}\end{array}\right):x,y\in \F_{q^4}\right\}.\]
is a linear cutting blocking set (while for $h$ even the result is no longer true) and by Theorem~\ref{thm:cutting}, is a $[8,4]_{q^{12}/q}$ rank-$3$-saturating system. So \[7\leq s_{q^{12}/q}(4,3)\leq 8,\]
for $q=2^h$ with $h$ odd. Note that, for $h$ even, $\mU$ may eventually be still a rank-$3$-saturating system in spite of the fact that Theorem \ref{thm:cutting} is not applicable. It would be interesting to know whether such an example of dependence on $q$ exists also for saturating systems.
\textcolor{black}{Note that from the dual distance bound, $\mU$ is a rank-$\rho$-saturating $[8,4]_{q^4/q}$ system, with $\rho\leq 5-d_{\rk}(\mC)$, where $\mC$ is a code associated with $\mU$.
In particular, if $\mC$ has minimum rank distance 2, $\mU$ is $\rho$-rank-saturating, with $\rho\leq 3$.
A parity-check matrix for $\mC$ (up to equivalence) is given by 
\cite[Proposition 4.14]{neriMRD}:
\[
   H=\left(
     \begin{array}{cc}
         0 & b \\
         b & 0 \\
         b^{q^2} & b^{q^2}+b^{q^3} \\
         b^{q^3} & b^{q^2}
     \end{array}
   \right),
\]
where $b \in \F^4_{q^4}$ has $\F_q$-rank equal to $4$ and $b^{q^j}:=(b^{q^i}_1,\dots,b^{q^i}_4)$ for each $i$.
From this it is easy to see that no word of $\F_{q}^8$ is contained in the nullspace of $H$ and hence $\mC$ has minimum distance at least 2.}

Finally, in \cite{neriMRD} it is shown that if $[t,k]_{q^m/q}$ is a linear cutting blocking set, then one can construct a $[t+m,k+1]_{q^m/q}$ linear cutting blocking set. In our terms, by Theorem~\ref{thm:cutting} we get that if a $[t,k]_{q^m/q}$ linear cutting blocking set exists, then
\[s_{q^{m(k-1)}/q}(k,k-1)\leq t \quad \text{ and } \quad s_{q^{mk}/q}(k+1,k)\leq t+m.\]
\end{remark}

\subsection{A construction from subgeometries}

In this subsection, we outline a construction that exploits the properties of particular subgeometries of $\mathrm{PG}(k-1,q^m)$, i.e. those arising from subfields of $\F_{q^m}$.

For the purposes of exposition, we start with a special case, which will serve as an example of a more general construction.

\begin{proposition}\label{prop:construction_rho_4}
Let $\F_{q^2}=\F_q[\alpha]$. For $k\geq 3$, the $[2k-3,k]_{q^4/q}$ system $\mathcal{U}$ 
defined by
\[\mathcal{U}=\left\{\left(\begin{array}{c}u\\\hline w \end{array}\right) : u\in\F_{q}^3,w\in\F_{q^2}^{k-3}\right\},\]
which has an associated generator matrix given by:
\[G=\left[ \begin{array}{c|c|c}
I_3 & \mathbf{0}   & \mathbf{0}        \\
\hline
\mathbf{0} & I_{k-3} & \alpha I_{k-3} \end{array}\right],\]
is rank-$3$-saturating. In particular, we have:
\[s_{q^4/q}(k,3)\leq 2k-3.\]
\end{proposition}

\begin{proof}
Fix $\beta_1,\beta_2\in \F_{q^4}$ such that $\F_{q^4}=\F_{q^2}+\langle\beta_1,\beta_2\rangle_{\F_q}$. 
For any $w\in\F_{q^4}$, write $w=\pi_{\beta_1}(w)\beta_1+\pi_{\beta_2}(w)\beta_2+\pi_{\F_{q^2}}(w)$ for $\pi_{\beta_1}(w),\pi_{\beta_2}(w)\in\F_q$ and $\pi_{\F_{q^2}}(w)\in\F_{q^2}$.

Consider a vector $v=(v_1,\ldots,v_k)^T\in\F_{q^4}^k$; we will show that $v=\lambda^{(1)} u^{(1)}+\lambda^{(2)} u^{(2)}+\lambda^{(3)}u^{(3)}$ for some $\lambda^{(1)},\lambda^{(2)},\lambda^{(3)}\in\F_{q^4}$ and $u^{(1)},u^{(2)},u^{(3)}\in\mathcal{U}$.
We first define the following functions: 
\[
\begin{split}
\varphi_1:\ \F_{q^4}\times\F_{q^4} &\longrightarrow \F_{q}\\
(x_1,x_2)&\longmapsto 
\left\{ \begin{array}{cl}
    \pi_{\beta_1}(x_1)^{-1}\pi_{\beta_1}(x_2) & \text{ if } \pi_{\beta_1}(x_1) \neq 0, \\
    0 & \text{ otherwise; } 
\end{array}
\right.
\end{split}
\]
and $\varphi_2:\ \F_{q^4}\times\F_{q^4}\times \F_{q^4} \longrightarrow \F_{q^2}$,
where 
\[
\varphi_2(x_1,x_2,x_3):=
\frac{\pi_{\F_{q^2}}(x_2)-\pi_{\F_{q^2}}(x_1)\varphi_1(x_1,x_2)}{\pi_{\beta_2}(x_2)-\pi_{\beta_2}(x_1)\varphi_1(x_1,x_2)}(\pi_{\beta_2}(x_3)-\pi_{\beta_2}(x_1)\varphi_1(x_1,x_3)),
\]
if $\pi_{\beta_2}(v_2)\ne \pi_{\beta_2}(v_1)\varphi_1(v_1,v_2)$ and $\varphi_2(x_1,x_2,x_3):=0$, otherwise.

We will first suppose that the following hold: 
\begin{itemize}
    \item[(I)]$\pi_{\beta_1}(v_1)\ne 0$,
    \item[(II)] $\pi_{\beta_2}(v_2)\ne \pi_{\beta_2}(v_1)\varphi_1(v_1,v_2)$,
    \item[(III)] $\pi_{\F_{q^2}}(v_3)\ne\pi_{\F_{q^2}}(v_1)\varphi_1(v_1,v_3)+\varphi_2(v_1,v_2,v_3)$.
\end{itemize}
Let
\[
\lambda^{(1)}:=\frac{v_1}{\pi_{\beta_1}(v_1)},
\]

\[
u^{(1)}:=\left(\begin{array}{c}
	\pi_{\beta_1}(v_1)\\\,\pi_{\beta_1}(v_2)\\ \pi_{\beta_1}(v_3)\\\hline \pi_{\beta_1}(v_4)\\\vdots\\ \pi_{\beta_1}(v_k)
\end{array}\right),
\]

\[
\lambda^{(2)}:=\beta_2+\frac{\pi_{\F_{q^2}}(v_2)-\pi_{\F_{q^2}}(v_1)\varphi_1(v_1,v_2)}{\pi_{\beta_2}(v_2)-\pi_{\beta_2}(v_1)\varphi(v_1,v_2)}=\frac{v_2-\lambda^{(1)}u^{(1)}_2-\pi_{\beta_1}(v_2-\lambda^{(1)}u^{(1)}_2)\beta_1}{\pi_{\beta_2}(v_2-\lambda^{(1)}u^{(1)}_2)},
\]

\[
u^{(2)}:=\left(\begin{array}{c}
0\\ \pi_{\beta_2}(v_2)-\pi_{\beta_2}(v_1)\varphi_1(v_1,v_2)\\ \pi_{\beta_2}(v_3)-\pi_{\beta_2}(v_1)\varphi_1(v_1,v_3)\\
\hline \pi_{\beta_2}(v_4)-\pi_{\beta_2}(v_1)\varphi_1(v_1,v_4)\\ \vdots\\ 
\pi_{\beta_2}(v_k)-\pi_{\beta_2}(v_1)\varphi_1(v_1,v_4)
\end{array}\right) = \left(\begin{array}{c}
	0\\\,\pi_{\beta_2}(v_2)\\ \pi_{\beta_2}(v_3)\\\hline \pi_{\beta_2}(v_4)\\\vdots\\ \pi_{\beta_2}(v_k)
\end{array}\right)- \left(\begin{array}{c}
    0\\
    \pi_{\beta_2}(\lambda^{(1)}u^{(1)}_2)\\
    \pi_{\beta_2}(\lambda^{(1)}u^{(1)}_3)\\
    \hline
    \pi_{\beta_2}(\lambda^{(1)}u^{(1)}_4)\\
    \vdots\\
    \pi_{\beta_2}(\lambda^{(1)}u^{(1)}_k)\\
\end{array}\right),
\]

\[
\begin{split}
\lambda^{(3)}&:=\pi_{\F_{q^2}}(v_3)-\pi_{\F_{q^2}}(v_1)\varphi_{1}(v_1,v_3)-\varphi_2(v_1,v_2,v_3)\\
&=v_3-\lambda^{(1)}u^{(1)}_3-\lambda^{(2)}u^{(2)}_3-\pi_{\beta_1}(v_3-\lambda^{(1)}u^{(1)}_3-\lambda^{(2)}u^{(2)}_3)\beta_1-\pi_{\beta_2}(v_3-\lambda^{(1)}u^{(1)}_3-\lambda^{(2)}u^{(2)}_3)\beta_2
\end{split}
\]

\[
\begin{split}
u^{(3)}:=&\left(\begin{array}{c}
	0\\0\\ 1\\\hline 0\\\vdots\\ 0
\end{array}\right)+\frac{1}{\lambda^{(3)}}
\left(\begin{array}{c}
	0\\0\\ 0\\\hline \pi_{\F_{q^2}}(v_4)-\pi_{\F_{q^2}}(v_1)\varphi_{1}(v_1,v_4)-\varphi_2(v_1,v_2,v_4)\\
	\vdots\\
	\pi_{\F_{q^2}}(v_k)-\pi_{\F_{q^2}}(v_1)\varphi_{1}(v_1,v_k)-\varphi_2(v_1,v_2,v_k)\\
\end{array}\right)\\
=& \left(\begin{array}{c}
	0\\0\\ 1\\\hline 0\\\vdots\\ 0
\end{array}\right)+\frac{1}{\lambda^{(3)}}
\left(\begin{array}{c}
	0\\0\\ 0\\\hline \pi_{\F_{q^2}}(v_4)\\\vdots\\ \pi_{\F_{q^2}}(v_k)
\end{array}\right)-\frac{1}{\lambda^{(3)}}
\left(\begin{array}{c}
	0\\0\\ 0\\\hline \pi_{\F_{q^2}}(\lambda^{(1)}u^{(1)}_4)\\\vdots\\ \pi_{\F_{q^2}}(\lambda^{(1)}u^{(1)}_k)
\end{array}\right)-\frac{1}{\lambda^{(3)}}
\left(\begin{array}{c}
	0\\0\\ 0\\\hline \pi_{\F_{q^2}}(\lambda^{(2)}u^{(2)}_4)\\\vdots\\ \pi_{\F_{q^2}}(\lambda^{(2)}u^{(2)}_k)
\end{array}\right).
\end{split}
\]

Direct computations show that $v=\lambda^{(1)}u^{(1)}+\lambda^{(2)}u^{(2)}+\lambda^{(3)}u^{(3)}$. Since $u^{(1)},u^{(2)}\in\F_{q}^k\subseteq\mU$ and $u^{(3)}\in\F_{q}^3\times\F_{q^2}^{k-3}\subseteq\mU$, we have that $\rho_{\mathrm{\rk}}(\mU)\le3$.

We now consider the possibility that one or more of the assumptions (I)-(III) do not hold. We will show that the argument holds with some minor modifications.

\begin{itemize}
    \item[(I)] Suppose that $\pi_{\beta_1}(v_1)=0$. 
\begin{itemize}
    \item[{\rm (a)}] If there exists an index $i\in\{2,\dots,k\}$ such that $\pi_{\beta_1}(v_i)\ne0$, repeat the passages written above replacing  $v_1$ with $v_i$.
    \item[{\rm (b)}] Otherwise, if there does not exist any $i\in\{2,\dots,k\}$ such that $\pi_{\beta_1}(v_i)\ne0$, set $\lambda^{(1)}=v_1$, $u^{(1)}=(1,0,\dots,0)^T$, and replace $\pi_{\beta_1}({v_1})^{-1}$ with the value zero in the formula for $\varphi_1$.
\end{itemize}
\item[(II)]Suppose that $\pi_{\beta_2}(v_2)=\pi_{\beta_2}(v_1)\varphi_1(v_1,v_2)$. 
\begin{itemize}
    \item[{\rm (a)}] If there exists an index $i\in\{3,\dots,k\}$ such that \[\pi_{\beta_2}(v_i)\ne\pi_{\beta_2}(v_1)\varphi_1(v_1,v_i),\] repeat the passages written above replacing $v_2$ with $v_i$.
    \item[{\rm (b)}] Otherwise, if there does not exist any $i\in\{3,\dots,k\}$ such that \[\pi_{\beta_2}(v_i)\ne\pi_{\beta_2}(v_1)\varphi_1(v_1,v_i),\] then set $\lambda^{(2)}=v_2-\lambda^{(1)}u_2^{(1)}$, $u^{(2)}=(0,1,\dots,0)^T$, 
    replace
    $$(\pi_{\beta_2}(v_2)-\pi_{\beta_2}(v_1)\varphi_1(v_1,v_3))^{-1}$$ with the value zero in the determination of $\lambda^{(3)},u^{(3)}$.
\end{itemize}
\item[(III)]Suppose that that $\pi_{\F_{q^2}}(v_3)\ne\pi_{\F_{q^2}}(v_1)\varphi_1(v_1,v_3)+\varphi_2(v_1,v_2,v_3)$.
\begin{itemize}
    \item[{\rm (a)}] If there exists an index $i\in\{4,\dots,k\}$ such that \[\pi_{\F_{q^2}}(v_i)\ne\pi_{\F_{q^2}}(v_1)\varphi_1(v_1,v_i)+\varphi_2(v_1,v_2,v_i),\] then replace $v_3$ with $v_i$ in the determination of $\lambda^{(3)}$, $u^{(3)}$.
    \item[{\rm (b)}] Otherwise, if there does not exist any $i\in\{4,\dots,k\}$ such that \[\pi_{\F_{q^2}}(v_i)\ne\pi_{\F_{q^2}}(v_1)\varphi_1(v_1,v_i)+\varphi_2(v_1,v_2,v_i),\] then the process has already terminated in the Step (II) and it is enough to set $\lambda^{(3)}=0$ and $u^{(3)}=(0,\dots,0)^T$.
\end{itemize}
\end{itemize}

\noindent To conclude the proof, we show that $\mU$ is exactly $3$-saturating. Let $\gamma_1,\gamma_2,\gamma_3\in\F_{q^4}$ be linearly independent over $\F_q$, and let $\overline{v}=(\gamma_1,\gamma_2,\gamma_3,0,\dots,0)^T\in\mU$. Due to the linear independence of the $\gamma_i$ over $\F_q$, it is not possible to saturate $\overline{v}$ with fewer than $3$ elements of $\mU$.
\end{proof}

The idea of the previous proof above, which is reminiscent of the Gram-Schmidt algorithm, allows us to obtain a construction which generalizes Proposition \ref{prop:construction_rho_4}. 

\begin{theorem}\label{thm:subgeo}
Let $r,t\geq 2$ and  $\F_{q^t}=\F_q[\alpha]$, for some $\alpha$, a root of an irreducible polynomial of degree $t$ over $\F_q$. For $h\geq 0$, the $[th + (r-1)t + 1,h+(r-1)t + 1]_{q^{rt}/q}$ system $\mathcal{U}$ defined by:
\[\mathcal{U}=\left\{\left(\begin{array}{c}u\\\hline w \end{array}\right) : u\in\F_{q}^{(r-1)t+1},w\in\F_{q^t}^h\right\},\]
which has an associated generator matrix given by:
\[G=\left[ \begin{array}{c|c|c|c|c}
I_{(r-1)t+1} & \mathbf{0}   & \mathbf{0}  & \mathbf{0} & \mathbf{0}      \\
\hline
\mathbf{0} & I_h & \alpha I_h & \dots & \alpha^{t-1} I_h \end{array}\right],\]
is rank-$((r-1)t+1)$-saturating. In particular,
\[s_{q^{rt}/q}(h+(r-1)t+1,(r-1)t+1)\leq th+(r-1)t+1.\]
\end{theorem}

\begin{proof}
Let $\{\beta_1,\dots,\beta_{(r-1)t}\}\subseteq \F_{q^{rt}}$ such that 
$\F_{q^{rt}} = \F_{q^t} + \langle \beta_1,\dots,\beta_{(r-1)t} \rangle_{\Fq}$.
For any $a\in\F_{q^{rt}}$, write $a=\sum_{j\in [(r-1)t]}\pi_{\beta_j}(a)\beta_j+\pi_{\F_{q^t}}(a)$ for $\pi_{\beta_j}(a)\in\F_q$ and $\pi_{\F_{q^t}}(a)\in\F_{q^t}$.
Let $k=h+(r-1)t+1$ and consider a vector $v=(v_1,\ldots,v_k)^T\in\F_{q^{rt}}^k$; we will show that $v=\sum_{j \in [(r-1)t+1]}\lambda^{(j)} u^{(j)}$ for some $\lambda^{(j)}\in\F_{q^{rt}}$ and 
$u^{(j)}\in\mathcal{U}$.
Suppose first that $\pi_{\beta_1}(v_1)\neq 0$. Define the following: 
\[
\begin{split}
\lambda^{(1)}:=
\frac{v_1}{\pi_{\beta_1}(v_1)}, \qquad\qquad &
u^{(1)}:=
\left(\begin{array}{c}
\pi_{\beta_1}(v_1)\\\pi_{\beta_1}(v_2)\\\vdots\\ \pi_{\beta_1}(v_{(r-1)t+1})\\ \hline \pi_{\beta_1}(v_{(r-1)t+2})\\\vdots\\ \pi_{\beta_1}(v_{h + (r-1)t + 1})
    \end{array}\right).
\end{split}
\]
If $\pi_{\beta_1}(v_1)=0$, then proceed similarly as described in (I) of Proposition \ref{prop:construction_rho_4}.\\
Now recursively define $\lambda^{(\ell)}$ and $u^{(\ell)}$ as follows:
for $\ell\in\{2,\ldots,(r-1)t\}$, we set
\[
\lambda^{(\ell)}:=
\frac{v_{\ell}-\sum_{i\in [\ell-1]}(\lambda^{(i)}u^{(i)}_\ell)-\sum_{j\in [\ell-1]}\pi_{\beta_j}(v_{\ell}-\sum_{i\in [j]}\lambda^{(i)}u^{(i)}_\ell)\beta_j}{\pi_{\beta_\ell}(v_{\ell}-\sum_{i\in [\ell-1]}\lambda^{(i)}u^{(i)}_\ell)},
\]

\[
u^{(\ell)}:=
\left(\begin{array}{c}
\mathbf{0}_{(\ell-1)\times 1}\\\pi_{\beta_\ell}(v_\ell)\\ \pi_{\beta_\ell}(v_{\ell+1})\\\vdots\\ \pi_{\beta_\ell}(v_{(r-1)t+1})\\ \hline \pi_{\beta_\ell}(v_{(r-1)t+2})\\\vdots\\ \pi_{\beta_\ell}(v_{h + (r-1)t + 1})
    \end{array}\right)
-
\sum_{i\in [\ell-1]}
    \left(\begin{array}{c}
\mathbf{0}_{(\ell-1)\times 1}\\\pi_{\beta_\ell}(\lambda^{(i)}u^{(i)}_{\ell}) \\\pi_{\beta_\ell}(\lambda^{(i)}u^{(i)}_{\ell+1}) \\\vdots\\\pi_{\beta_\ell}(\lambda^{(i)}u^{(i)}_{(r-1)t+1})\\ \hline \pi_{\beta_\ell}(\lambda^{(i)}u^{(i)}_{(r-1)t+2})\\\vdots\\ \pi_{\beta_\ell}(\lambda^{(i)}u^{(i)}_{h + (r-1)t + 1})
    \end{array}\right),
\]
under the assumption that 
\begin{equation}\label{eq:cond1}
\pi_{\beta_\ell}(v_{\ell}-\sum_{i\in [\ell-1]}\lambda^{(i)}u^{(i)}_\ell)\ne0, \text{ for }\ell\in[(r-1)t].
\end{equation}
If (\ref{eq:cond1}) does not hold, then this means that $\lambda^{(\ell)}$ and $u^{(\ell)}$ are not necessary for the decomposition, and it is possible to proceed to the next step. We proceed in a manner similar to (II) of Proposition \ref{prop:construction_rho_4}.

Furthermore, define:
\[
\lambda^{((r-1)t+1)}:=
\pi_{\F_{q^t}}(v_{(r-1)t+1})-\sum_{i\in [(r-1)t]}\pi_{\F_{q^t}}(\lambda^{(i)}u^{(i)}_{(r-1)t+1}),
\]
 and
\[
u^{(r-1)t+1}:=
\left(\begin{array}{c}
\mathbf{0}_{((r-1)t)\times 1}\\ 1\\ \hline\\ \frac{\pi_{\F_{q^t}}(v_{(r-1)t+2})}{\lambda^{((r-1)t+1)}}\\ \vdots \\ \frac{\pi_{\F_{q^t}}(v_{h + t(s-1) + 1})}{\lambda^{((r-1)t+1)}}
    \end{array}\right)
-
\frac{1}{\lambda^{((r-1)t+1)}}\sum_{i\in [(r-1)t]}\left(\begin{array}{c}
\mathbf{0}_{((r-1)t+1)\times 1}\\ \hline \pi_{\F_{q^t}}(\lambda^{(i)}u^{(i)}_{(r-1)t+2})\\\vdots\\ \pi_{\F_{q^t}}(\lambda^{(i)}u^{(i)}_{h + t(s-1) + 1}) \end{array}\right),
\]
where we assume that 
\begin{equation}\label{eq:cond2}
\pi_{\F_{q^t}}(v_{(r-1)t+1})\ne\sum_{i\in [(r-1)t]}\pi_{\F_{q^t}}(\lambda^{(i)}u^{(i)}_{(r-1)t+1}).
\end{equation}
Finally, if (\ref{eq:cond2}) does not hold, then continue as in (III) of Proposition \ref{prop:construction_rho_4}.

In order to prove that $\mU$ is $((r-1)t+1)$-saturating, we show that the following hold:
\begin{itemize}

    \item[(i)]
    $v_\ell=\sum_{i\in [\ell]} \lambda^{(i)}u^{(i)}_\ell$ for $\ell\in[(r-1)t]$, 
    \item[(ii)]$\pi_{\beta_\ell}(v_k)=\pi_{\beta_\ell}\left(\sum_{i\in [\ell]}\lambda^{(i)}u^{(i)}_k\right)=\pi_{\beta_\ell}\left(\sum_{i\in [j]}\lambda^{(i)}u^{(i)}_k\right)$ for 
    $\ell\in[(r-1)t]$, $k\in\{\ell+1,\dots,th + t(s-1) + 1\}$, and $j\in \{\ell,\ldots,(r-1)t+1\}$,
    \item[(iii)] $\sum_{i\in [(r-1)t+1]}\lambda^{(i)}u^{(i)}=v$.
\end{itemize}

Direct computations show that
\[
\lambda^{(1)}u^{(1)}_1=\frac{v_1}{\pi_{\beta_1}(v_1)}\pi_{\beta_1}(v_1)=v_1,
\]
and that, furthermore: 
\[
\begin{split}
\sum_{i\in [\ell]}\lambda^{(i)}u^{(i)}_\ell=&\lambda^{(\ell)}u^{(\ell)}_\ell+\sum_{i\in [\ell-1]} (\lambda^{(i)}u^{(i)}_\ell)\\ 
=&\frac{v_{\ell}-\sum_{i\in [\ell-1]}(\lambda^{(i)}u^{(i)}_\ell)-\sum_{j\in [\ell-1]}\pi_{\beta_j}(v_{\ell}-\sum_{i\in [j]}\lambda^{(i)}u^{(i)}_\ell)\beta_j}{\pi_{\beta_\ell}(v_{\ell}-\sum_{i\in [\ell-1]}\lambda^{(i)}u^{(i)}_\ell)}\pi_{\beta_\ell}\left(v_{\ell}-\sum_{i\in [\ell-1]}\lambda^{(i)}u^{(i)}_\ell\right)\\
&+\sum_{i\in [\ell-1]} (\lambda^{(i)}u^{(i)}_\ell)\\
=&v_{\ell}-\sum_{i\in [\ell-1]} (\lambda^{(i)}u^{(i)}_\ell)+\sum_{i\in [\ell-1]} (\lambda^{(i)}u^{(i)}_\ell)\\
=&v_\ell.
\end{split}
\]
Hence (i) holds.

We now prove (ii), noting that by construction, $\pi_{\beta_i}(\lambda^{(i)})=1$, for $i\le (r-1)t$. Moreover, for $\ell\in[(r-1)t]$, it is straightforward to show that the second equality in (ii) holds, as $\lambda^{(i)}\in \langle\beta_i,\beta_{i+1},\dots,\beta_{(r-1)t}\rangle_{\F_q} + \F_{q^2}$ and $u^{(i)}\in\F_{q}^{th + (r-1)t + 1}$ for $i\le (r-1)t$. 
Firstly, we have that 
\[
\pi_{\beta_1}(\lambda^{(1)}u^{(1)}_k)=\frac{\pi_{\beta_1}(v_1)}{\pi_{\beta_1}(v_1)}u^{(1)}_1=\pi_{\beta_1}(v_1).
\]

Consider now $\ell\in\{2,\ldots,(r-1)t\}$. By construction, we have that: 
\[
\begin{split}
\pi_{\beta_\ell}\left(\sum_{i\in [\ell]}\lambda^{(i)}u^{(i)}_k\right)=
&\pi_{\beta_\ell}(\lambda^{(\ell)}u^{(\ell)}_k)+\pi_{\beta_\ell}\left(\sum_{i\in [\ell-1]} (\lambda^{(i)}u^{(i)})\right)\\ 
=&\pi_{\beta_\ell}(v_\ell)-\sum_{i\in [\ell-1]} \pi_{\beta_\ell}(\lambda^{(i)}u^{(i)}_k)+\pi_{\beta_\ell}\left(\sum_{i\in [\ell-1]} (\lambda^{(i)}u^{(i)}_k)\right)\\
=&\pi_{\beta_\ell}(v_\ell)-\sum_{i\in [\ell-1]} \pi_{\beta_\ell}(\lambda^{(i)}u^{(i)})_k+\sum_{i\in[\ell-1]} \pi_{\beta_\ell}(\lambda^{(i)}u^{(i)})\\
=&\pi_{\beta_\ell}(v_\ell),
\end{split}
\] 
which implies (ii).

To prove (iii) it remains to show that $\pi_{\F_{q^t}}(v_k)=\pi_{\F_{q^t}}\left(\sum_{i\in[(r-1)t+1]}\lambda^{(i)}u^{(i)}_k\right)$, for $k\in\{(r-1)t+1,\dots,th + t(s-1) + 1\}$. 

By construction, for $k\in\{(r-1)t+1,\ldots,th + t(s-1) + 1\}$ we have:
\[
\begin{split}
\pi_{\F_{q^t}}\left(\sum_{i\in[(r-1)t+1]}\lambda^{(i)}u^{(i)}_{k}\right)=&\lambda^{((r-1)t+1)}u^{((r-1)t+1)}_{k}+\sum_{i\in[(r-1)t]}\lambda^{(i)}u^{(i)}_{k}\\
=&\pi_{\F_{q^t}}(v_{(r-1)t+1})-\sum_{i\in [(r-1)t]}\pi_{\F_{q^t}}(u^{(i)}_k)+\sum_{i\in [(r-1)t]}\lambda^{(i)}u^{(i)}_k\\
=&\pi_{\F_{q^t}}(v_k).
\end{split}
\]

Since $u^{(\ell)}\in\F_{q}^{(r-1)t+h+1}\subseteq\mU$, for $\ell\in[(r-1)t]$, and $u^{((r-1)t+1)}\in\F_q^{(r-1)t}\times\F_{q^{t}}^{h}\subseteq\mU$ we have that $\rho_{\mathrm{rk}}(\mU)\le (r-1)t+1$.
Moreover, using the same argument as in the proof of Proposition~\ref{prop:construction_rho_4}, $\mU$ is exactly $((r-1)t+1)$-saturating as taking $\gamma_1,\dots,\gamma_{(r-1)t+1}\in\F_{q^{(r-1)t+1}}$ be linearly independent over $\F_q$, the vector $\overline{v}=(\gamma_1,\dots,\gamma_{(r-1)t+1},0,\dots,0)^T\in\mU$ cannot be saturated with fewer than $(r-1)t+1$ elements of $\mU$.

\end{proof}

\begin{remark}
Let us suppose $q>2$. Note that for $h\geq 0$, 
\[\frac{rth+((r-1)t+1)^2}{(r-1)t+1}\leq s_{q^{rt}/q}((r-1)t+1+h,(r-1)t+1)\leq th+(r-1)t+1\]
and the difference between the upper and the lower bound is
\[\frac{(r-1)t(t-1)}{(r-1)t+1}\cdot h = \left( t-1 -\frac{t-1}{(r-1)t+1} \right)h <(t-1) h\]
Note that, for $t=2$ and $h=1$, for all $r\geq 2$ the difference is strictly less than $1$, so that
$$s_{q^{2r}/q}(2r,2r-1)=2r+1$$
for $q>2$.

In the case $q=2$, with a similar argument we get 
$$2r\leq s_{2^{2r}/2}(2r,2r-1)\leq 2r+1.$$
Now, clearly $s_{q^{2r}/q}(2r,2r-1)$ cannot be equal to $2r$ (for any $q$), since the only $[2r,2r]_{q^m/q}$ code is the full space and in this case $\rho=2r$. So, $s_{2^{2r}/2}(2r,2r-1)=2r+1$.
\end{remark}

\section{Conclusion}

For the convenience of the reader, we summarize the main results on $s_{q^m/q}(k,\rho)$ proved in this paper. First,
by  Theorem \ref{thm_LowerBound} and Theorem \ref{thm:generalUB},
\begin{align*}
\left\lceil\frac{mk}{\rho}\right\rceil -m+\rho\leq  s_{q^m/q}(k,\rho)\leq m(k-\rho)+\rho, & & \text{for $q>2$ and $\rho>1$,}\\
\left\lceil\frac{mk-1}{\rho}\right\rceil -m+\rho\leq  s_{q^m/q}(k,\rho)\leq m(k-\rho)+\rho, & & \text{for $q=2$ and $\rho>1$,}\\
s_{q^m/q}(k,1)=m(k-1)+1, & & \text{for all $q$}.
\end{align*}
The Monotonicity Theorem (Theorem \ref{th:ineq}) and the Direct Sum Theorem (Theorem \ref{prop:DirectSum}) state that, for all positive integers $m,k,k'$, $\rho\in[\min\{k,m\}]$, $\rho'\in[\min\{k',m\}$,
\begin{itemize}
    \item[{\rm (a)}] If $\rho<\min\{k,m\}$, then $s_{q^m/q}(k,\rho+1) \leq s_{q^m/q}(k,\rho).$
    \item[{\rm (b)}] $s_{q^m/q}(k,\rho)<s_{q^m/q}(k+1,\rho)$.
    \item[{\rm (c)}] If $\rho<m$, then $s_{q^m/q}(k+1,\rho+1) \leq s_{q^m/q}(k,\rho)+1$.
    \item[{\rm (d)}] If $\rho+\rho'\leq \min\{k+k',m\}$, $s_{q^m/q}(k+k',\rho+\rho') \leq s_{q^m/q}(k,\rho)+s_{q^m/q}(k',\rho')$.
\end{itemize}
The upper bound is sharpened for particular cases: for every $r,k\geq 2$, thanks to the construction using linear cutting blocking sets (Corollary \ref{coro:lincutsat}) we have
\[s_{q^{r(k-1)}/q}(k,k-1)\leq 2k+r-2.\]
Using subgeometries (Theorem \ref{thm:subgeo}), for $t,s\geq 2$ and $h\geq 0$ we have:
\[s_{q^{rt}/q}((r-1)t+1+h,(r-1)t+1)\leq th+(r-1)t+1.\]
Finally, we list some cases for which $s_{q^m/q}(k,\rho)$ is determined, namely:
\begin{align*}
    s_{q^m/q}(k,1) & = m(k-1)+1, & \text{ for all }m,k\geq 2,  \\
    s_{q^m/q}(k,k) & = k, & \text{ for all }m,k\geq 2,\\
    s_{q^{2r}/q}(3,2) & = r+2, & \text{ for } r\neq 3,5\bmod 6\text{ and }r\geq 4,
    \\
    s_{q^{2r}/q}(3,2) & = r+2, & \text{ for } \gcd(r,(q^{2s}-q^s+1)!)=1, r\text{ odd}, 1\leq s\leq r, \gcd(r,s)=1,
    \\
    s_{q^{10}/q}(3,2) & = 7, & \text{ for } q=p^{15h+s}, p\in\{2,3\}, \gcd(s,15)=1,\\
     s_{q^{10}/q}(3,2) & = 7, & \text{ for } q=5^{15h+1},\\
    s_{q^{10}/q}(3,2) & = 7, & \text{ for } q \text{\ odd}, q= 2,3\bmod 5 \text{ and for } q=2^{2h+1}, h\geq 1,
    \\
    s_{q^{2r}/q}(2r,2r-1)  &=2r+1, & \text{ for all }r\geq 2.
\end{align*}

\section*{Acknowledgments}

The results of this paper are the result of a collaboration that arose within the IRC-PHC Ulysses project ``Geometric Constructions of Codes for Secret Sharing Schemes''. The research of the first author was partially supported by the Irish Research Council, grant n. GOIPD/2020/597. The second author was partially supported by the ANR-21-CE39-0009 - BARRACUDA (French \emph{Agence Nationale de la Recherche}).
The authors express their deep gratitude to Ferdinando Zullo and Julien Lavauzelle for the inspiring discussions on the subject of the paper. The authors wish to thank the anonymous reviewers for their meticulous reading of this manuscript and whose comments greatly improved this work.



\section*{Data Availability Statement}
Data sharing not applicable to this article as no datasets were generated or analysed during the current study.

\bibliographystyle{abbrv}
\bibliography{references.bib}

\end{document}